\documentclass{amsart}
\usepackage[T1]{fontenc}
\usepackage[latin9]{inputenc}
\usepackage[a4paper]{geometry}
\usepackage{amsthm}
\usepackage{amstext}
\usepackage{amssymb}
\usepackage{setspace}
\usepackage{amscd}
\usepackage{mathptmx}
\usepackage{bbm}
\usepackage{hyperref}
\makeatletter

\numberwithin{equation}{section}
\numberwithin{figure}{section}
\theoremstyle{plain}
\newtheorem{thm}{\protect\theoremname}[section]
  \theoremstyle{remark}
  \newtheorem*{rem*}{\protect\remarkname}
  \theoremstyle{plain}
  \newtheorem{prop}[thm]{\protect\propositionname}
  \theoremstyle{plain}
  \newtheorem{cor}[thm]{\protect\corollaryname}
  \theoremstyle{remark}
  \newtheorem*{acknowledgement*}{\protect\acknowledgementname}
  \theoremstyle{definition}
  \newtheorem{defn}[thm]{\protect\definitionname}
  \theoremstyle{remark}
  \newtheorem*{notation*}{\protect\notationname}
  \theoremstyle{plain}
  \newtheorem{fact}[thm]{\protect\factname}
  \theoremstyle{definition}
  \newtheorem{problem}[thm]{\protect\problemname}
  \theoremstyle{plain}
  \newtheorem{lem}[thm]{\protect\lemmaname}
  \theoremstyle{remark}
  \newtheorem{rem}[thm]{\protect\remarkname}
  
  \providecommand{\acknowledgementname}{Acknowledgement}
  \providecommand{\corollaryname}{Corollary}
  \providecommand{\definitionname}{Definition}
  \providecommand{\factname}{Fact}
  \providecommand{\lemmaname}{Lemma}
  \providecommand{\notationname}{Notation}
  \providecommand{\problemname}{Problem}
  \providecommand{\propositionname}{Proposition}
  \providecommand{\remarkname}{Remark}
\providecommand{\theoremname}{Theorem}

\newcommand{\1}{\mathbbm{1}}
\newcommand{\N}{\mathbb{N}}
\newcommand{\F}{\mathbb{F}}

\newcommand{\R}{\mathbb{R}}

\renewcommand{\Pi}{\pi}
\renewcommand{\emptyset}{\varnothing}
\renewcommand{\hat}{\widehat}

\newcommand\diam{\mathrm{diam}}

\DeclareMathOperator*{\Int}{Int}
\newcommand\id{\mathrm{id}}

\DeclareMathOperator*{\card}{card}

\DeclareMathOperator*{\Con}{Con}

\DeclareMathOperator*{\Aut}{Aut}
\newcommand{\emptyword}{\emptyset}
\newcommand{\e}{\mathrm{e}}
\DeclareMathOperator*{\Lur}{\Lambda_{ur}}
\DeclareMathOperator*{\Lr}{\Lambda_{r}}

\begin{document}

\title{Fractal Models for Normal Subgroups of Schottky Groups}

\author{Johannes Jaerisch}

\thanks{The author was supported by the research fellowship JA 2145/1-1 of
the German Research Foundation (DFG)}

\address{Department of Mathematics, Graduate School of Science Osaka University,
1-1 Machikaneyama Toyonaka, Osaka, 560-0043 Japan }

\email{jaerisch@cr.math.sci.osaka-u.ac.jp}


\begin{abstract}
For a normal subgroup $N$ of the free group $\F_{d}$ with at least
two generators we introduce the radial limit set $\Lr(N,\Phi)$
of $N$ with respect to a graph directed Markov system $\Phi$ associated
to $\F_{d}$. These sets are shown to provide fractal models of radial
limit sets of normal subgroups of Kleinian groups of Schottky type.
Our main result states that if $\Phi$ is symmetric and linear, then we have that $\dim_{H}(\Lr(N,\Phi))=\dim_{H}(\Lr(\F_d,\Phi))$
if and only if the quotient group $\F_{d}/N$ is amenable, where $\dim_{H}$ denotes the Hausdorff dimension. This extends
a result of Brooks for normal subgroups of Kleinian groups to a large
class of fractal sets. Moreover, we show that if $\F_{d}/N$ is non-amenable
then $\dim_{H}(\Lr(N,\Phi))>\dim_{H}(\Lr(\F_d,\Phi))/2$,
which extends results by Falk and Stratmann and by Roblin.
\end{abstract}
\maketitle

\section{Introduction and Statement of Results}

In this paper we introduce and investigate linear models for the Poincar\'{e}
series and  the radial limit set  of    normal subgroups of  Kleinian groups of Schottky type. Here, a linear model means a linear graph directed Markov system (GDMS) associated to the free group  $\F_{d}=\langle g_1,\dots, g_d \rangle$ on $d\ge2$ generators. Precise definitions are given in  Section \ref{sub:Graph-Directed-Markov}, but briefly, such a system $\Phi$ is given by the vertex set $V:=\{ g_1,g^{-1}_1\dots, g_d,g^{-1}_d\}$,  edge set $E:=\{ (v,w)\in V^2:v\neq w^{-1}\}$ and by a family of contracting similarities $\{ \phi_{\left(v,w\right)} : \left(v,w\right)\in E \}$ of the  Euclidean space $\R^d$, for $d\ge1$, such that for each  $\left(v,w\right)\in E$
the contraction ratio of the similarity  $\phi_{\left(v,w\right)}$ is independent
of $w$. We  denote this ratio by $c_{\Phi}(v)$.  Also, we say that $\Phi$ is \emph{symmetric
}if $c_{\Phi}\left(g\right)=c_{\Phi}\left(g^{-1}\right)$ for all
$g\in V$. In order to state our first two main results, we must also make two  further definitions. For this,  we extend $c_{\Phi}$ to a function $c_{\Phi}:\F_{d}\rightarrow \R$ by setting $c_{\Phi}\left(g\right):=\prod_{i=1}^{n}c_{\Phi}\left(v_{i}\right)$,
where $n\in \N$ and  $\left(v_{1},\dots,v_{n}\right)\in V^{n}$ refers to  the unique representation of $g$ as a reduced word. Also, for each subgroup $H$ of $\F_d$, we  introduce the \emph{Poincar\'{e}
series of $H$ }and the \emph{exponent of convergence of $H$ with
respect to $\Phi$}  which are defined for $s\ge 0$ by
\[
P\left(H,\Phi,s\right):=\sum_{h\in H}\left(c_{\Phi}\left(h\right)\right)^{s} \quad  \text{ and } \quad\delta\left(H,\Phi\right):=\inf\left\{ t\ge 0:P\left(H,\Phi,t\right)<\infty\right\} .
\]
Our first main result gives a relation between amenability and the exponent of convergence. 

\begin{thm}
\label{thm:lineargdms-amenability-dichotomy}Let $\Phi$ be a symmetric
linear GDMS associated to $\F_{d}$. For every  normal
subgroup $N$ of $\F_{d}$, we have that
\[
\delta\left(\F_{d},\Phi\right)=\delta\left(N,\Phi\right)\,\,\textrm{if and only if }\,\,\F_{d}/N\textrm{ is amenable}.
\]
\end{thm}
Our second main result gives a lower bound for the exponent of convergence $\delta(N,\Phi)$.
\begin{thm}
\label{thm:lineargdms-lowerhalfbound}Let $\Phi$ be a symmetric linear
GDMS associated to $\F_{d}$. For every non-trivial normal subgroup
$N$ of $\F_{d}$, we  have that
\[
\delta\left(N,\Phi\right)>\delta\left(\F_{d},\Phi\right)/2.
\]

\end{thm}

Our next results study certain limit sets which provide fractal models of radial limit sets of Kleinian groups.    More precisely, for a GDMS $\Phi$ associated to $\F_{d}$ and a  subgroup $H$ of $\F_{d}$,   we will consider  the   \emph{radial limit set }$\Lr(H,\Phi)$ \emph{of $H$} and the  \emph{uniformly radial limit set }$\Lur(H,\Phi)$ \emph{of $H$ with respect to $\Phi$} (see  Definition \ref{def:gdms-associated-to-freegroup-and-radiallimitsets}).

\begin{prop}
\label{pro:lineargdms-brooks}Let $\Phi$ be a linear GDMS associated
to $\F_{d}$. For every  normal subgroup $N$ of $\F_{d}$,
we have that
\[
\delta\left(N,\Phi\right)=\dim_{H}\left(\Lur(N,\Phi)\right)=\dim_{H}\left(\Lr(N,\Phi)\right).
\]

\end{prop}
The following  corollary is an immediate consequence of Theorem
\ref{thm:lineargdms-amenability-dichotomy}, Theorem \ref{thm:lineargdms-lowerhalfbound}
and Proposition \ref{pro:lineargdms-brooks}.
\begin{cor}
\label{main-cor}
Let $\Phi$ be a symmetric linear GDMS associated to $\F_{d}$. For every normal subgroup $N$ of $\F_{d}$, we have that
\[
\dim_{H}\left(\Lr(N,\Phi)\right)=\dim_{H}\left(\Lr(\F_d,\Phi)\right)\mbox{ if and only if }\F_{d}/N\mbox{ is amenable}.
\]
Moreover, if $N$ is non-trivial, then we have that
\[
\dim_{H}\left(\Lr(N,\Phi)\right)\,>\,\dim_{H}\left(\Lr(\F_d,\Phi)\right)/2.
\]

\end{cor}

Let us now briefly summarize the corresponding results for normal subgroups of Kleinian groups, which served as  the motivation for our  main results in this paper. A more
detailed discussion of Kleinian groups and how these relate to the concept of
a GDMS will be given in Section \ref{sec:Kleinian-groups}. We start by giving a short introduction to Kleinian groups.

Recall that, for $m\in\N$, an $\left(m+1\right)$-dimensional hyperbolic  manifold  can be described
by the hyperbolic $\left(m+1\right)$-space $\mathbb{D}^{m+1}:=\left\{ z\in\R^{m+1}:\left|z\right|<1\right\} $
equipped with the hyperbolic metric $d$ and  quotiented by the action
of a Kleinian group $G$. The \emph{Poincar\'{e}
series  of   $G$} and the  \emph{exponent
of convergence  of
$G$} are for $s\ge 0$ given by
\[
P\left(G,s\right):=\sum_{g\in G}\e^{-sd\left(0,g\left(0\right)\right)} \quad  \text{ and }\quad \delta\left(G\right):=\inf\left\{ t\ge 0:P\left(G,t\right)<\infty\right\}.
\]
A normal subgroup $N$ of a Kleinian group $G$ gives
rise to an intermediate  covering of the associated hyperbolic manifold $\mathbb{D}^{m+1}/G$.
It was shown by Brooks in  \cite{MR783536}
that if $N$ is  a normal subgroup of a convex cocompact Kleinian group
$G$ such that $\delta\left(G\right)>m/2$, then  we have that
\begin{equation}
\delta\left(N\right)=\delta\left(G\right)\mbox{ if and only if }G/N\mbox{ is amenable.}\label{eq:intro-brooks}
\end{equation}

Moreover, Falk and Stratmann  \cite{MR2097162} showed that for
every non-trivial normal subgroup $N$ of a non-elementary  Kleinian
group $G$ we have $\delta\left(N\right)\ge\delta\left(G\right)/2$.
Using different methods, Roblin (\cite{MR2166367}) proved that if
$G$ is of $\delta\left(G\right)$-divergence type, that is, if $P\left(G,\delta\left(G\right)\right)=\infty$,
then we have
\begin{equation}
\delta\left(N\right)>\delta\left(G\right)/2.\label{eq:intro-roblin}
\end{equation}
Another proof of (\ref{eq:intro-roblin}) can be found in \cite{Bonfert-Taylor2012}
for a convex cocompact Kleinian group $G$, where it was also shown that $\delta(N)$ can be arbitrarily close to $\delta(G)/2$.

Note that our results stated in Theorem \ref{thm:lineargdms-amenability-dichotomy} and Theorem \ref{thm:lineargdms-lowerhalfbound} extend the assertions given in (\ref{eq:intro-brooks}) and (\ref{eq:intro-roblin}) for Kleinian groups.

\begin{rem*}
Note that  in Theorem \ref{thm:lineargdms-amenability-dichotomy} there is no restriction on $\delta\left(\F_{d},\Phi\right)$
  whereas for the proof of
(\ref{eq:intro-brooks}) it was vital  to assume that $\delta\left(G\right)>m/2$.
It was conjectured by Stratmann \cite{MR2191250} that this assumption can be removed from
Brooks' Theorem. In fact, it was shown by
Sharp in \cite[Theorem 2]{MR2322540} that if $G$ is a finitely generated Fuchsian groups, that is for $m=1$, and if $N$ is a normal subgroup of $G$, then  amenability of $G/N$ implies $\delta\left(G\right)=\delta\left(N\right)$.
Recently, Stadlbauer \cite{Stadlbauer11} showed  that the equivalence
in (\ref{eq:intro-brooks}) extends to the class of essentially free
Kleinian groups with arbitrary exponent of convergence $\delta\left(G\right)$.
\end{rem*}

Finally, let us turn our attention to limit sets of Kleinian groups. For a Kleinian group $G$,  the \emph{radial limit set $L_{\mathrm{r}}\left(G\right)$} and
the \emph{uniformly radial limit set $L_{\mathrm{ur}}\left(G\right)$}
(see Definition \ref{def:radiallimitsets-fuchsian}) are both subsets
of the boundary $\mathbb{S}:=\left\{ z\in\R^{m+1}:\left|z\right|=1\right\} $
of $\mathbb{D}^{m+1}$. By a theorem of Bishop and Jones (\cite[Theorem 1.1]{MR1484767},
cf. \cite{MR2087134}),  we have for every Kleinian group $G$ that
\begin{equation}
\delta\left(G\right)=\dim_{H}\left(L_{\mathrm{ur}}\left(G\right)\right)=\dim_{H}\left(L_{\mathrm{r}}\left(G\right)\right),\label{eq:bishop-jones}
\end{equation}
where $\dim_{H}$ denotes the Hausdorff dimension with respect to
the Euclidean metric on $\mathbb{S}$. Combining (\ref{eq:intro-brooks})
and (\ref{eq:bishop-jones}) then shows that for every 
normal subgroup $N$ of a convex cocompact Kleinian group $G$ for which
$\delta\left(G\right)>m/2$, we have
\begin{equation}
\dim_{H}\left(L_{\mathrm{r}}\left(N\right)\right)=\dim_{H}\left(L_{\mathrm{r}}\left(G\right)\right)\mbox{ if and only if }G/N\mbox{ is amenable.}\label{eq:brooks-via-hausdorffdimension}
\end{equation}

We would like to point that there is a  close analogy between the results on  radial limit sets of Kleinian groups stated  in (\ref {eq:bishop-jones}) and  (\ref{eq:brooks-via-hausdorffdimension}),  and our results in the context of linear  GDMSs associated to free groups  stated  in Proposition \ref{pro:lineargdms-brooks} and Corollary \ref{main-cor}.

Let us now further clarify the relation between GDMSs associated to free groups and Kleinian groups of Schottky type  (see Definition \ref{def:kleinian-of-schottkytype}).  For this, recall that a Kleinian group of Schottky type  $G=\langle g_{1},\dots,g_{d}\rangle$ is isomorphic to a free group. In Definition \ref{def:canonical-model-kleinianschottky} we introduce
 a  \emph{canonical GDMS $\Phi_{G}$} \emph{associated to} $G$. We will then show in Proposition \ref{pro:canonicalgdms-gives-radiallimitset} that for every non-trivial normal subgroup $N$ of $G$ we have that
\[
L_{\mathrm{r}}\left(N\right)=\Lr(N,\Phi_G)\mbox{ and }L_{\mathrm{ur}}\left(N\right)=\Lur(N,\Phi_G).
\]
This shows that our fractal models of radial limit sets of Kleinian groups of Schottky type can be thought of as a replacement of the conformal  generators of the Kleinian group by similarity maps. Our main results  show that several important properties of Kleinian groups  extend to these fractal models.

Let us now end this introductory section by briefly summarizing the methods used to obtain our results and how this paper is organized.
Theorem \ref{thm:lineargdms-amenability-dichotomy} and
Theorem \ref{thm:lineargdms-lowerhalfbound} are based on and extend
results of Woess \cite{MR1743100} and Ortner and Woess \cite{MR2338235},
which in turn refer back to work of P\'{o}lya \cite{MR1512028} and Kesten \cite{MR0109367,MR0112053}.
Specifically, we provide generalizations of \cite{MR2338235} for
weighted graphs. Our new thermodynamic formalism for group-extended
Markov systems (see Section \ref{sec:Thermodynamic-Formalism-grpextension})
 characterizes amenability of discrete groups in terms of topological
pressure and the spectral radius of the Perron-Frobenius operator
acting on a certain $L^{2}$-space.

The paper is organized as follows. In Section \ref{sec:Preliminaries}
we collect the necessary background on thermodynamic formalism, GDMSs and
random walks on graphs. In Section \ref{sec:Thermodynamic-Formalism-grpextension}, we prove  a thermodynamic formalism for group-extended Markov systems,
which is also of independent interest. Using the results
of Section \ref{sec:Thermodynamic-Formalism-grpextension} we prove
our main results in Section \ref{sec:Proofs}. Finally, in Section \ref{sec:Kleinian-groups} we
provide the background on Kleinian groups of Schottky type, which has motivated our results.

After having finished this paper,  Stadlbauer (\cite{Stadlbauer11}) proved a 
partial  extension of Theorem \ref{thm:amenability-dichotomy-markov} (see
Remark \ref{proof-comment-stadlabuer}).
Moreover, in \cite{Jaerisch12a} the author has extended Lemma \ref{lem:delta-half-divergencetype}
and Theorem \ref{thm:lineargdms-lowerhalfbound} in order to give
a short new proof of (\ref{eq:intro-roblin}) for  Kleinian groups.
\begin{acknowledgement*}
Parts of this paper constitute certain  parts of the author's doctoral
thesis supervised by Marc Kesseb\"ohmer at the University of Bremen.
The author would like to express his  deep gratitude to Marc Kesseb\"ohmer
and Bernd Stratmann for their support and many fruitful discussions. The author thanks an anonymous referee for the careful reading of the manuscript and for valuable comments on the exposition of this paper.  Final thanks go to Sara Munday for helping to improve the presentation of the paper significantly.
\end{acknowledgement*}

\section{Preliminaries\label{sec:Preliminaries}}

\subsection{Symbolic Thermodynamic Formalism}

Throughout, the underlying symbolic space for  the symbolic thermodynamic formalism will be a 
\emph{Markov shift $\Sigma$ }, which is given by
\[
\Sigma:=\left\{ \omega:=\left(\omega_{1},\omega_{2},\ldots\right)\in I^{\N}:\; a\left(\omega_{i},\omega_{i+1}\right)=1\,\,\mbox{for all }i\in\N\right\} ,
\]
where $I$ denotes a finite or countable infinite \emph{alphabet}, the matrix $A=\left(a\left(i,j\right)\right)\in\left\{ 0,1\right\} ^{I\times I}$
is the \emph{incidence matrix} and the \emph{shift map} $\sigma:\Sigma\rightarrow\Sigma$
is defined by  $\sigma(\left(\omega_{1},\omega_{2},\ldots\right)):=\left(\omega_{2},\omega_{3},\ldots\right)$, for each $\left(\omega_{1},\omega_{2},\ldots\right)\in \Sigma$.
We always assume that for each $i\in I$ there exists $j\in I$ such
that $a\left(i,j\right)=1$. The set of \emph{$A$-admissible words} of length
$n\in\mathbb{N}$ is given by
\[
\Sigma^{n}:=\left\{(\omega_1, \dots, \omega_n)\in I^{n}:\,\, a\left(\omega_{i},\omega_{i+1}\right)=1\mbox{ for all }i\in\left\{ 1,\dots,n-1\right\} \right\}
\]
and we set $\Sigma^{0}:=\left\{ \emptyword\right\} $, where $\emptyword$
denotes the empty word. Note that $\emptyset$ will also be used to denote the
empty set. The set of all finite $A$-admissible words is
denoted by
\[
\Sigma^{*}:=\bigcup_{n\in\N}\Sigma^{n}.
\]
Let us also define the \emph{word length function} $\left|\cdot\right|:\,\Sigma^{*}\cup\Sigma\cup\left\{ \emptyword\right\} \rightarrow\N_{0}\cup\left\{ \infty\right\} $,
where for $\omega\in\Sigma^{*}$ we set $\left|\omega\right|$ to
be the unique $n\in\N$ such that $\omega\in\Sigma^{n}$, for $\omega\in\Sigma$
we set $\left|\omega\right|:=\infty$ and $\emptyword$ is the unique
word of length zero. For each $\omega\in\Sigma^{*}\cup\Sigma\left\{ \emptyword\right\} $
and $n\in\N_{0}$ with $n\le\left|\omega\right|$, we define $\omega_{|n}:=\left(\omega_{1},\dots,\omega_{n}\right)$.
For $\omega,\tau\in\Sigma$, we set $\omega\wedge\tau$ to be the longest common initial block of $\omega$ and $\tau$, that is, $\omega\wedge\tau:=\omega_{|l}$,
where $l:=\sup\left\{ n\in\N_{0}:\omega_{|n}=\tau_{|n}\right\} $.
For $\omega\in\Sigma^{n}$, $n\in\N_{0}$, the \emph{cylinder set} $[\omega]$ defined by  $\omega$ is given by $\left[\omega\right]:=\left\{ \tau\in\Sigma:\tau_{|n}=\omega\right\} $.
Note that $\left[\emptyword\right]=\Sigma$.

If $\Sigma$ is the Markov shift with alphabet $I$ whose incidence
matrix consists entirely of $1$s, then we have that $\Sigma=I^{\N}$
and $\Sigma^{n}=I^{n}$, for all $n\in\N$. Then we set $I^{*}:=\Sigma^{*}$
and $I^{0}:=\left\{ \emptyword\right\} $. For $\omega,\tau\in I^{*}\cup\left\{ \emptyword\right\} $, let $\omega\tau\in I^{*}\cup\left\{ \emptyword\right\} $ denote 
the \emph{concatenation} of $\omega$ and $\tau$, which is defined
by $\omega\tau:=\left(\omega_{1},\dots,\omega_{\left|\omega\right|},\tau_{1},\dots,\tau_{\left|\tau\right|}\right)$, 
for $\omega,\tau\in I^{*}$,  and if  $\omega\in I^{*}\cup\left\{ \emptyword\right\} $ then we define  $\omega\emptyword:=\emptyword\omega:=\omega$. Note that $I^{*}$
is the free semigroup over the set $I$ which satisfies the following
universal property:  For each semigroup $S$ and for every map $u:I\rightarrow S$,
there exists a unique semigroup homomorphism $\hat{u}:I^{*}\rightarrow S$
such that $\hat{u}\left(i\right)=u\left(i\right)$, for all $i\in I$
(see \cite[Section 3.10]{MR1650275}).

Moreover, we equip $I^{\N}$ with the product topology of the discrete topology
on $I$ and the Markov shift $\Sigma\subset I^{\N}$ is equipped with
the subspace topology. The latter  topology on $\Sigma$ is the weakest topology on $\Sigma$ 
such that for each $j\in\N$ the \emph{canonical projection on the
$j$-th coordinate} $p_{j}:\Sigma\rightarrow I$ is continuous. A
countable basis for this topology on $\Sigma$ is given by the cylinder
sets $\left\{ \left[\omega\right]:\omega\in\Sigma^{*}\right\} $.
We will  use  the following metric  generating the topology
on $\Sigma$. For $\alpha>0$ fixed, we define the metric $d_{\alpha}$
on $\Sigma$ given by
\[
d_{\alpha}\left(\omega,\tau\right):=\e^{-\alpha\left|\omega\wedge\tau\right|},\mbox{ for all }\omega,\tau\in\Sigma.
\]

For a function $f:\Sigma\rightarrow\R$ and $n\in\N_{0}$,  we use the notation $S_{n}f:\Sigma\rightarrow\R$
to denote the \emph{ergodic sum} of $f$ with respect to the left-shift map $\sigma$, in other words, $S_{n}f:=\sum_{i=0}^{n-1}f\circ\sigma^{i}$.

Furthermore,  the following function spaces will be crucial throughout.
\begin{defn}
We say that a function $f:\Sigma\rightarrow\R$ is {\em bounded} whenever $\Vert f\Vert_{\infty}:=\sup_{\omega\in\Sigma}\left|f\left(\omega\right)\right|$ is finite.
We denote by $C_{b}\left(\Sigma\right)$ the real vector space of
bounded continuous functions on $\Sigma$. We say that $f:\Sigma\rightarrow\R$
is \emph{$\alpha$-H\"older continuous}, for some $\alpha>0$, if
\[
V_{\alpha}\left(f\right):=\sup_{n\ge1}\left\{ V_{\alpha,n}\left(f\right)\right\} <\infty,
\]
where for each $n\in\N$ we let
\[
V_{\alpha,n}\left(f\right):=\sup\left\{ \e^{-\alpha}\frac{\left|f\left(\omega\right)-f\left(\tau\right)\right|}{d_{\alpha}\left(\omega,\tau\right)}:\omega,\tau\in\Sigma,\left|\omega\wedge\tau\right|\ge n\right\} .
\]
The function \emph{$f$ }is called \emph{ H\"older continuous} if there exists
$\alpha>0$ such that $f$ is $\alpha$-H\"older continuous.

For $\alpha>0$ we also introduce the real vector space
\[
H_{\alpha}\left(\Sigma\right):=\left\{ f\in C_{b}\left(\Sigma\right):\, f\textrm{ is }\alpha-\textrm{H\"older continuous}\right\} ,
\]
which we assume to be equipped with the norm $\Vert \cdot \Vert_{\alpha}$ which is given by
\[
\Vert f\Vert_{\alpha}:=\Vert f\Vert_{\infty}+V_{\alpha}\left(f\right).
\]
\end{defn}
We need the following notion of pressure, which was originally introduced in \cite[Definition 1.1]{JaerischKessebohmer10}.
\begin{defn}
\label{def:induced-topological-pressure}For
$\varphi,\psi:\Sigma\rightarrow\R$ with $\psi\ge0$, $\mathcal{C}\subset\Sigma^{*}$
and $\eta>0$, the $\psi$\emph{-induced pressure of} $\varphi$ (with
respect to $\mathcal{C}$) is given by
\[
\mathcal{P}_{\psi}\left(\varphi,\mathcal{C}\right):=\limsup_{T\rightarrow\infty}\frac{1}{T}\log\sum_{{\omega\in\mathcal{C}\atop T-\eta<S_{\omega}\psi\le T}}\exp S_{\omega}\varphi,
\]
where we have set $S_{\omega}\varphi:=\sup_{\tau\in\left[\omega\right]}S_{\left|\omega\right|}\varphi\left(\tau\right)$. Note that $\mathcal{P}_{\psi}\left(\varphi,\mathcal{C}\right)$ is an element of   $\overline{\R}:=\R\cup\left\{-\infty, +\infty\right\} $. \end{defn}
\begin{rem*}
It was shown in \cite[Theorem 2.4]{JaerischKessebohmer10} that the
definition of $\mathcal{P}_{\psi}\left(\varphi,\mathcal{C}\right)$
is in fact independent of the choice of $\eta>0$. For this reason, 
we do not refer to $\eta>0$ in the definition of the induced pressure. \end{rem*}
\begin{notation*}
If $\psi$ and/or $\mathcal{C}$ is left out in the notation
of induced pressure, then we tacitly assume that $\psi=1$ and/or
$\mathcal{C}=\Sigma^{*}$, that is, we let   $\mathcal{P}(\varphi):=\mathcal{P}_{1}\left(\varphi,\Sigma^*\right)$.
\end{notation*}
The following fact is taken from \cite[Remark 2.11, Remark 2.7]{JaerischKessebohmer10}.
\begin{fact}
\label{fac:criticalexponents-via-pressure}Let  $\Sigma$ be a  Markov
shift over a finite alphabet. If $\varphi,\psi:\Sigma\rightarrow\R$ are two functions such that 
$\psi\ge c>0$, for some $c>0$, and if  $\mathcal{C}\subset\Sigma^{*}$
then $\mathcal{P}_{\psi}\left(\varphi,\mathcal{C}\right)$ is equal to the
unique real number $s\in\R$ for which  $\mathcal{P}\left(\varphi-s\psi,\mathcal{C}\right)=0$.
Moreover, we have that
\[
\mathcal{P}_{\psi}\left(\varphi,\mathcal{C}\right)=\inf\left\{ s\in\R:\sum_{\omega\in\mathcal{C}}\e^{S_{\omega}\left(\varphi-s\psi\right)}<\infty\right\} .
\]

\end{fact}
The next definition goes back to the work of Ruelle and Bowen (\cite{MR0289084,bowenequilibriumMR0442989}).
\begin{defn}
Let $\varphi:\Sigma\rightarrow\R$ be continuous. We say
that a Borel probability measure \emph{$\mu$ is a Gibbs measure for
$\varphi$ }if there exists a constant $C>0$ such that
\begin{equation}
C^{-1}\le\frac{\mu\left[\omega\right]}{\e^{S_{\left|\omega\right|}\varphi\left(\tau\right)-\left|\omega\right|\mathcal{P}\left(\varphi\right)}}\le C,\mbox{ for all }\omega\in\Sigma^{*}\mbox{ and }\tau\in\left[\omega\right].\label{eq:gibbs-equation}
\end{equation}

\end{defn}
The Perron-Frobenius operator, which we are going to define now,  provides  a useful tool for guaranteeing  the
existence of Gibbs measures and for deriving some of  the  stochastic
properties of these measures  (see \cite{MR0289084,bowenequilibriumMR0442989}).
\begin{defn}\label{def:perron-frobenius}
Let $\Sigma$ be a Markov shift over a finite alphabet  and let $\varphi:\Sigma\rightarrow \R$ be continuous. The \emph{Perron-Frobenius operator associated to $\varphi$} is the operator $\mathcal{L}_{\varphi}:C_b(\Sigma)\rightarrow C_b(\Sigma)$ which is given, for each $f\in C_b(\Sigma)$ and $x\in \Sigma$, by
\[
\mathcal{L}_{\varphi}(f)(x):=\sum_{y\in \sigma^{-1}\left\{ x\right\} }\e^{\varphi\left(y\right)}f\left(y\right).
\]
\end{defn}

The  following theorem  summarizes some of the main results of
the thermodynamic formalism for a Markov shift $\Sigma$ with a finite alphabet
$I$ (see for instance \cite{MR648108} and \cite[Section 2]{MR2003772}).
Here, $\Sigma$ is called  \emph{irreducible} if for all $i,j\in I$
there exists $\omega\in\Sigma^{*}\cup\left\{ \emptyword\right\} $
such that $i\omega j\in\Sigma^{*}$. Moreover, for $k\in\N_{0}$,  the $\sigma$-algebra generated by $\left\{ \left[\omega\right]:\omega\in\Sigma^{k}\right\} $ is denoted by $\mathcal{C}(k)$,   and we say that  $f:\Sigma \rightarrow \R$  is $\mathcal{C}(k)$ -\emph{measurable} if $f^{-1}(A)\in \mathcal{C}(k)$ for every $A\in \mathcal{B}\left(\R\right)$,  where $\mathcal{B}\left(\R\right)$ denotes the Borel $\sigma$-algebra on $\R$.
\begin{thm}
\label{thm:perron-frobenius-thm-urbanski}Let $\Sigma$ be an irreducible
Markov shift over a finite alphabet and let $\varphi:\Sigma\rightarrow\R$
be $\alpha$-H\"older continuous, for some $\alpha>0$. Then there exists
a unique Borel probability measure $\mu$ supported on $\Sigma$ such
that $\int\mathcal{L}_{\varphi}\left(f\right)\, d\mu=\e^{\mathcal{P}\left(\varphi\right)}\int f\, d\mu$, for all $f\in C_{b}\left(\Sigma\right)$. Furthermore, $\mu$ is a Gibbs
measure for $\varphi$ and  there exists a unique $\alpha$-H\"older
continuous function $h:\Sigma\rightarrow\R^{+}$ such that $\int h\,d\mu=1$ and $\mathcal{L}_{\varphi}\left(h\right)=\e^{\mathcal{P}\left(\varphi\right)}h$.
The measure $h\,d\mu$ is the unique $\sigma$-invariant Gibbs measure
for $\varphi$ and will be denoted by $\mu_{\varphi}$. If $\varphi:\Sigma\rightarrow\R$ is $\mathcal{C}(k)$-measurable, for some $k\in \N_0$,
then $h$ is $\mathcal{C}\!\left(\max\left\{ k-1,1\right\} \right)$-measurable.
\end{thm}

\subsection{Graph Directed Markov Systems\label{sub:Graph-Directed-Markov}}

In this section we will first recall the definition of a graph directed
Markov system (GDMS), which was introduced by Mauldin and Urba\'nski
\cite{MR2003772}. Subsequently, we will  introduce the notion of a linear  GDMS associated to a  free group and certain radial limit sets.
\begin{defn}
\label{def:gdms}
A \emph{graph directed Markov system (GDMS)} $\Phi:=\left(V,\left(X_{v}\right)_{v\in V},E,i,t,\left(\phi_{e}\right)_{e\in E},A\right)$
consists of a finite vertex set $V$, a family of nonempty compact
metric spaces $\left(X_{v}\right)_{v\in V}$, a countable edge set
$E$, the maps $i,t:E\rightarrow V$ defining the initial and terminal
vertex of an edge, a family of injective contractions $\phi_{e}:X_{t\left(e\right)}\rightarrow X_{i\left(e\right)}$
with Lipschitz constants bounded by some $0<s<1$, and an edge incidence
matrix $A=\left(a\left(e,f\right)\right)\in\left\{ 0,1\right\} ^{E\times E}$
such that $a\left(e,f\right)=1$ implies $t\left(e\right)=i\left(f\right)$,
for all $e,f\in E$. For a GDMS $\Phi$ there exists a canonical \emph{coding
map} $\pi_{\Phi}:\Sigma_{\Phi}\rightarrow\oplus_{v\in V}X_{v}$, which is defined by
\[
\bigcap_{n\in\N}\phi_{\omega_{|n}}\left(X_{t\left(\omega_{n}\right)}\right)=\left\{ \pi_{\Phi}\left(\omega\right)\right\},
\]
where $\oplus_{v\in V}X_{v}$ denotes the disjoint union of the sets
$X_{v}$, $\phi_{\omega|_n}:=\phi_{\omega_1}\circ\cdots\circ\phi_{\omega_n}$ and  $\Sigma_{\Phi}$ denotes the Markov shift with alphabet
$E$ and incidence matrix $A$. We set
\[
J\left(\Phi\right):=\pi_{\Phi}\left(\Sigma_{\Phi}\right),\quad J^{*}\left(\Phi\right):=\bigcup_{F\subset E,\card\left(F\right)<\infty}\pi_{\Phi}\left(\Sigma_{\Phi}\cap F^{\N}\right),
\]
and refer to $J\left(\Phi\right)$ as the \emph{limit set of $\Phi$}.
\end{defn}
The following notion was introduced in \cite[Section 4]{MR2003772}.
\begin{defn}
\label{def:cgdms}The GDMS $\Phi=\left(V,\left(X_{v}\right)_{v\in V},E,i,t,\left(\phi_{e}\right)_{e\in E},A\right)$
is called \emph{conformal }if the following conditions are satisfied.

\renewcommand{\theenumi}{\alph{enumi}}
\begin{enumerate}
\item \label{enu:cgdms-a-phasespace}For $v\in V$, the \emph{phase space}
$X_{v}$ is a compact connected subset of a Euclidean space $\left(\R^{D},\Vert\cdot\Vert\right)$,
for some $D\geq1$, such that $X_{v}$ is equal to the closure of
its interior, that is $X_{v}=\overline{\Int(X_{v})}$.
\item \textit{\emph{\label{enu:cgdms-b-osc}(}}\textit{Open set condition
}\textit{\emph{(OSC))}} For all $a,b\in E$ with $a\ne b$, we have
that
\[
\phi_{a}\left(\Int(X_{t\left(a\right)})\right)\cap\phi_{b}\left(\Int(X_{t\left(b\right)})\right)=\emptyset.
\]

\item \label{enu:cgdms-c-conformalextension}For each vertex $v\in V$ there
exists an open connected set $W_{v}\supset X_{v}$ such that the map
$\phi_{e}$ extends to a $C^{1}$ conformal diffeomorphism of $W_{v}$
into $W_{i\left(e\right)}$, for every $e\in E$ with $t\left(e\right)=v$.
\item \textit{\emph{\label{enu:cgdms-d-coneproperty}(}}\textit{Cone property}\textit{\emph{)}}
There exist $l>0$ and $0<\gamma<\pi/2$ such that for each $x\in X\subset\R^{D}$
there exists an open cone $\Con(x,\gamma,l)\subset\Int(X)$ with vertex
$x$, central angle of measure $\gamma$ and altitude $l$.
\item \label{enu:cgdms-e-hoelderderivative}There are two constants $L\geq1$
and $\alpha>0$ such that for each $e\in E$ and $x,y\in X_{t\left(e\right)}$
we have
\[
\big|\left|\phi_{e}'(y)\right|-\left|\phi_{e}'(x)\right|\big|\leq L\inf_{u\in W_{t\left(e\right)}}\left|\phi_{e}'\left(u\right)\right|\Vert y-x\Vert^{\alpha}.
\]

\end{enumerate}
\end{defn}
The \emph{associated geometric potential $\zeta_{\Phi}:\Sigma_{\Phi}\rightarrow\R^{-}$
of a conformal GDMS $\Phi$} is defined by
\[
\zeta_{\Phi}\left(\omega\right):=\log\left|\phi_{\omega_{1}}'\left(\pi_{\Phi}\left(\sigma\left(\omega\right)\right)\right)\right|,\mbox{ for all }\omega\in\Sigma_{\Phi}.
\]
A Markov shift $\Sigma$   with a finite or countable alphabet $I$  is called \emph{finitely irreducible}
if there exists a finite set $\Lambda\subset\Sigma^{*}$ such that
for all $i,j\in I$ there exists a word $\omega\in\Lambda\cup\left\{ \emptyword\right\} $
such that $i\omega j\in\Sigma^{*}$ (see \cite[Section 2]{MR2003772}). Note that if  $I$ is     finite, then $\Sigma$ is finitely irreducible if and only if $\Sigma$ is irreducible. 
The following result from \cite[Theorem 3.7]{MR2413348} shows that
in the sense of Hausdorff dimension, the limit set of a conformal
GDMS with a finitely irreducible incidence matrix can be exhausted
by its finitely generated subsystems. The last equality in
Theorem \ref{thm:cgdms-bowen-formula}  follows from \cite[Corollary 2.10]{JaerischKessebohmer10}
since the associated geometric potential  of the conformal GDMS  $\Phi$  is bounded away from zero by $-\log\left(s\right)$, where $s$ denotes the uniform bound of the Lipschitz constants of the contractions of $\Phi$ (see Definition \ref{def:gdms}).
\begin{thm}
[Generalized Bowen's formula]\label{thm:cgdms-bowen-formula}Let
$\Phi$ be a conformal GDMS such that $\Sigma_{\Phi}$ is finitely
irreducible. We then have that
\[
\dim_{H}\left(J\left(\Phi\right)\right)=\dim_{H}\left(J^{*}\left(\Phi\right)\right)=\inf\left\{ s\in\R:\mathcal{P}\left(s\zeta_{\Phi}\right)\le0\right\} =\mathcal{P}_{-\zeta_{\Phi}}\left(0,\Sigma_{\Phi}^{*}\right).
\]

\end{thm}
Let us now give the definition of a GDMS $\Phi$ associated
to the free group $\F_{d}$ of rank $d\geq2$ and introduce the radial limit set of a normal subgroup
$N$ of $\F_{d}$ with respect to $\Phi$.
\begin{defn}
\label{def:gdms-associated-to-freegroup-and-radiallimitsets} Let $\Phi=\left(V,\left(X_{v}\right)_{v\in V},E,i,t,\left(\phi_{e}\right)_{e\in E},A\right)$ be a GDMS and let $d\ge2$. The GDMS  $\Phi$ is said to be   \emph{associated to $\F_{d}=\langle g_{1},\dots,g_{d}\rangle$}, if $V=\left\{ g_1, g_{1}^{-1},\dots,g_d, g_{d}^{-1}\right\} $,
$E=\left\{ \left(v,w\right)\in V^{2}:v\neq w^{-1}\right\} $, the maps $i,t:E\rightarrow V$
are  given by $i\left(v,w\right)=v$ and $t\left(v,w\right)=w$,  for each $(v,w)\in E$, and
the incidence matrix $A=\left(a\left(e,f\right)\right)\in\left\{ 0,1\right\} ^{E\times E}$
satisfies $a\left(e,f\right)=1$ if and only if $t\left(e\right)=i\left(f\right)$,
for all $e,f\in E$. If additionally $\Phi$ is a conformal GDMS such that, for each $(v,w)\in E$, the map  $\phi_{\left(v,w\right)}$ is a similarity for which the contraction ratio  is independent
of $w$, then $\Phi$ is called a    \emph{linear GDMS associated to} $\F_d$.

For a  subgroup $H$ of $\F_{d}$ and a GDMS $\Phi$ associated
to $\F_{d}$, the\emph{ radial }and the\emph{ uniformly radial limit
set of $H$ with respect to $\Phi$} are respectively  given by
\begin{align*}
\Lr(H,\Phi) & :=\pi_{\Phi}\left\{ \left(v_{i},w_{i}\right)\in\Sigma_{\Phi}:\exists\gamma\in\F_{d}\mbox{ such that for infinitely many }n\in\N,\, v_{1}\cdot\dots\cdot v_{n}\in H\gamma\right\}
\\
\text{and}
\\
\Lur(H,\Phi) & :=\pi_{\Phi}\left\{ \left(v_{i},w_{i}\right)\in\Sigma_{\Phi}:\exists\Gamma\subset\F_{d}\mbox{ finite such that for all }n\in\N,\, v_{1}\cdot\dots\cdot v_{n}\in H \Gamma \right\} .
\end{align*}

\end{defn}

\begin{rem*}It is clear that if $\Phi$ is  a GDMS generated by a family of similarity maps, then $\Phi$ automatically satisfies (c) and (e) in Definition \ref{def:cgdms} of a conformal GDMS.
\end{rem*}

\subsection{Random Walks on Graphs and Amenability}

In this section we collect some useful definitions and results concerning random walks on graphs. We will mainly follow \cite{MR1743100}.
\begin{defn}
\label{def:graphs}A \emph{graph $X=\left(V,E\right)$} consists of
a countable \emph{vertex set} $V$ and an \emph{edge set }$E\subset V\times V$
such that $\left(v,w\right)\in E$ if and only if $\left(w,v\right)\in E$.
We write $v\sim w$ if $\left(v,w\right)\in E$, which defines an
equivalence relation on $V$. For all $v,w\in V$ and $k\in\N_{0}$,
a \emph{path of length $k$ from $v$ to $w$} is a sequence $\left(v_{0},\dots,v_{k}\right)\in V^{k+1}$
such that $v_{0}=v$, $v_{k}=w$ and $v_{i-1}\sim v_{i}$ for all
$1\le i\le k$. For all $v\in V$, let $\mathrm{deg}\left(v\right):=\card\left\{ w\in V:w\sim v\right\} $
denote the \emph{degree }of the vertex $v$. The graph $\left(V,E\right)$
is called \emph{connected} if for all $v,w\in V$ with $v\neq w$,
there exists $k\in\N$ and a path of length $k$ from $v$ to $w$.
For a connected graph $X=\left(V,E\right)$ and $v,w\in V$ we let
$d_{X}\left(v,w\right)$ denote the minimal length of all paths from
$v$ to $w$, which defines the \emph{graph metric }$d_{X}\left(\cdot,\cdot\right):V\times V\rightarrow\N_{0}$.
The graph $\left(V,E\right)$ is said to have \emph{bounded geometry}
if it is connected and if $\sup_{v\in V}\left\{ \mathrm{deg}\left(v\right)\right\} <\infty$.
For each set of vertices $A\subset V$ we define $dA:=\left\{ v\in A:\exists w\in V\setminus A\mbox{ such that }v\sim w\right\} $.
\end{defn}
We now recall an important property of groups, which was introduced
by von Neumann \cite{vonNeumann1929amenabledef} under the German
name {\em messbar}. Later, groups with this property were renamed {\em amenable
groups} by Day \cite{day1949amenabledef} and also referred to as {\em groups with full
Banach mean value} by F\o lner \cite{MR0079220}.
\begin{defn}
\label{def:amenable-group}A discrete group\emph{
$G$ }is said to be\emph{ amenable} if there exists a finitely additive probability
measure $\nu$ on the set of all subsets of $G$ which is invariant
under left multiplication by elements of $G$, that is, $\nu\left(A\right)=\nu\left(g\left(A\right)\right)$
for all $g\in G$ and $A\subset G$.
\end{defn}
We will also require  the concept of an amenable graph,  which extends the concept of amenability for  groups (see Proposition \ref{pro:groupamenable-iff-graphamenable}
below).
\begin{defn}
\label{def:amenable-graph}A graph $X=\left(V,E\right)$
with bounded geometry is called \emph{amenable} if and only if there
exists $\kappa>0$ such that for all finite sets $A\subset V$ we
have $\card\left(A\right)\le\kappa\card\left(dA\right)$.
\end{defn}
For the study of graphs in terms of amenability, the following definition
is useful.
\begin{defn}
A \emph{rough isometry (or quasi-isometry) between
two metric spaces} $\left(Y,d_{Y}\right)$ and $\left(Y',d_{Y'}\right)$
is a map $\varphi:Y\rightarrow Y'$ which has the following properties.
There exist constants $A,B>0$ such that for all $y_{1},y_{2}\in Y$
we have
\[
A^{-1}d_{Y}\left(y_{1},y_{2}\right)-A^{-1}B\le d_{Y'}\left(\varphi\left(y_{1}\right),\varphi\left(y_{2}\right)\right)\le Ad_{Y}\left(y_{1},y_{2}\right)+B
\]
and for all $y'\in Y'$ we have
\[
d_{Y'}\left(y',\varphi\left(Y\right)\right)\le B.
\]
Two metric spaces \emph{$\left(Y,d_{Y}\right)$ }and\emph{ $\left(Y',d_{Y'}\right)$
}are said to be \emph{roughly isometric} if there exists a rough isometry between
$\left(Y,d_{Y}\right)$ and $\left(Y',d_{Y'}\right)$. For connected
graphs $X=\left(V,E\right)$ and $X=\left(V',E'\right)$ with graph
metrics $d_{X}$ and $d_{X'}$ we say that the graphs $X$ and $X'$\emph{
}are\emph{ roughly isometric} if the metric spaces $\left(V,d_{X}\right)$
and $\left(V',d_{X'}\right)$ are roughly isometric.
\end{defn}
The next theorem states that amenability of graphs is invariant under rough isometries
(\cite[Theorem 4.7]{MR1743100}).
\begin{thm}
\label{thm:amenability-is-roughisometry-invariant}Let $X$ and $X'$
be graphs with bounded geometry such that $X$ and $X'$ are roughly
isometric. We then have that
$X$ is  amenable if and only if $X'$ is amenable.

\end{thm}
The Cayley graph of a group provides the connection between groups
and graphs.
\begin{defn}
We say that a set $S\subset G$ is a \emph{symmetric
set of generators of the group} $G$ if $\left\langle S\right\rangle =G$
and if  $g^{-1}\in S$, for all  $g\in S$.  For a group
$G$ and a symmetric set of generators $S$, the \emph{Cayley graph
of $G$ with respect to $S$} is the graph with vertex set $G$ and
edge set $E:=\left\{ \left(g,g'\right)\in G\times G:g^{-1}g'\in S\right\} $.
We denote this graph by $X\left(G,S\right)$.
\end{defn}
Next proposition shows that amenability of groups and graphs is compatible
(\cite[Proposition 12.4]{MR1743100}).
\begin{prop}
\label{pro:groupamenable-iff-graphamenable}A finitely generated group
$G$ is amenable if and only if one (and hence  every) Cayley graph $X\left(G,S\right)$
of $G$ with respect to a finite symmetric set of generators $S\subset G$
is amenable.
\end{prop}
Let us now relate amenability of graphs to spectral properties
of transition operators.
\begin{defn}
For a finite or countably infinite discrete vertex
set $V$, we say that the matrix $P=\left(p\left(v,w\right)\right)\in\R^{V\times V}$
is a \emph{transition matrix on $V$} if $p\left(v,w\right)\ge0$
and $\sum_{u\in V}p\left(v,u\right)=1$, for all $v,w\in V$. A Borel
measure $\nu$ supported on $V$ is\emph{ $P$-invariant }if we have
$\sum_{u\in V}\nu\left(u\right)p\left(u,w\right)=\nu\left(w\right)$,
for all $w\in V$.
\end{defn}
The following definitions introduce the concept of a transition matrix to be adapted to a graph (see \cite[(1.20, 1.21)]{MR1743100}).
\begin{defn}
\label{def:uniformly-irred-bounded-range}For
a connected graph $X=\left(V,E\right)$ and a transition matrix $P=\left(p\left(v,w\right)\right)\in\R^{V\times V}$
on $V$, we say that \emph{$P$ is uniformly irreducible with respect
to $X$} if there exist $K\in\N$ and $\epsilon>0$ such that
for all $v,w\in V$ satisfying $v\sim w$ there exists $k\in\N$ with
$k\le K$ such that $p^{\left(k\right)}\left(v,w\right)\ge\epsilon$.
We say that $P$ has \emph{bounded range with respect to $X$} if
there exists $R>0$ such that $p\left(v,w\right)=0$ whenever $d_{X}\left(v,w\right)>R$.
\end{defn}
Let $P=\left(p\left(v,w\right)\right)\in\R^{V\times V}$ be a transition
matrix on $V$ with $P$-invariant Borel measure $\nu$ on $V$. It
is well-known that $P$ defines a linear operator on $\ell^{2}\left(V,\nu\right)$
through the equations
\[
Pf\left(v\right):=\sum_{w\in V}p\left(v,w\right)f\left(w\right),\quad\mbox{for all }v\in V\mbox{ and }f\in\ell^{2}\left(V,\nu\right)
\]
and that the norm of this operator is less or equal to one. For the
spectral radius $\rho\left(P\right)$ of the operator $P$ on $\ell^{2}\left(V,\nu\right)$
we cite the following result from \cite{MR2338235}. This result has
a rather long history going back to \cite{MR0109367,MR0112053} (see
also \cite{MR0159230,MR678175,MR743744,MR894523,MR938257,MR943998,MR1245225,MR1743100}).
\begin{thm}
[Ortner, Woess]\label{thm:woess-amenability-randomwalk-characterization}Let
$X=\left(V,E\right)$ be a graph with bounded geometry and let $P$
denote a transition matrix on $V$ such that $P$ is uniformly irreducible
with respect to $X$ and has bounded range with respect
to $X$. If there exists a $P$-invariant Borel measure $\nu$ on
$V$ and a constant $C\ge1$ such that $C^{-1}\le\nu\left(w\right)\le C$,
for all $w\in V$, then we have that $\rho\left(P\right)=1$  if and only if $X$ is amenable.

\end{thm}

\section{Thermodynamic Formalism for Group-extended Markov Systems\label{sec:Thermodynamic-Formalism-grpextension}}

Throughout this section our setting is as follows.
\begin{enumerate}
\item $\Sigma$ is a Markov shift with finite alphabet $I$ and left-shift map
$\sigma:\Sigma\rightarrow\Sigma$.
\item $G$ is a countable discrete group $G$ with Haar measure (counting
measure) $\lambda.$
\item $\Psi:I^{*}\rightarrow G$ is a semigroup homomorphism such that the
following property holds. For all $a,b\in I$ there exists $\gamma\in\Psi^{-1}\left\{ \id\right\} \cap\Sigma^{*}\cup\left\{ \emptyword\right\} $
such that $a\gamma b\in\Sigma^{*}$.
\item $\varphi:\Sigma\rightarrow\R$ denotes a H\"older continuous potential
with $\sigma$-invariant Gibbs measure $\mu_{\varphi}$, $\mathcal{L}_{\varphi}:C_b(\Sigma)\rightarrow C_b(\Sigma) $ denotes the Perron-Frobenius operator associated to $\varphi$, and $h:\Sigma \rightarrow  \R$ denotes the unique H\"older continuous eigenfunction of $\mathcal{L}_{\varphi}$ with corresponding eigenvalue $\e^{\mathcal{P}(\varphi)}$ whose existence is guaranteed by Theorem \ref{thm:perron-frobenius-thm-urbanski}.
\end{enumerate}
In this section we will address the following problem. \begin{problem}
\label{mainproblem}How is  amenability of $G$ reflected in the relationship
between $\mathcal{P}\left(\varphi,\Psi^{-1}\left\{ \id\right\} \cap\Sigma^{*}\right)$
and $\mathcal{P}\left(\varphi\right)$?
\end{problem}
It turns out that in order to investigate Problem \ref{mainproblem}
it is helpful to consider group-extended Markov systems (defined below), which were
studied in (\cite{MR1803461,MR1906436}) for certain abelian groups.
\begin{defn}
\label{def:skew-product-dynamics}The
skew-product dynamics on  $\left(\Sigma\times G,\sigma\rtimes\Psi\right)$, for which  the transformation $\sigma\rtimes\Psi:\Sigma\times G\rightarrow\Sigma\times G$ is given by
\[
\left(\sigma\rtimes\Psi\right)\left(\omega,g\right):=\left(\sigma\left(\omega\right),g\Psi\left(\omega_{1}\right)\right),\ \text{ for all }
\left(\omega,g\right)\in\Sigma\times G,
\]
is called a \emph{group-extended Markov system}. We let $\pi_{1}:\Sigma\times G\rightarrow\Sigma\mbox{ and }\pi_{2}:\Sigma\times G\rightarrow G$
denote the projections to the first and to the second factor of $\Sigma\times G$. \end{defn}
\begin{rem*}
Throughout, we assume that  $\Sigma\times G$ is equipped with the product topology. Note that by item
(3) of our standing assumptions we have that the group-extended Markov
system $\left(\Sigma\times G,\sigma\rtimes\Psi\right)$ is topologically
transitive if and only if $\Psi\left(\Sigma^{*}\right)=G$.
\end{rem*}

\subsection{Perron-Frobenius Theory\label{sub:Perron-Frobenius-Theory}}

In this section,  we investigate the relationship
between the pressure $\mathcal{P}\left(\varphi,\Psi^{-1}\left\{ \id\right\} \cap\Sigma^{*}\right)$
and the spectral radius of a Perron-Frobenius operator associated
to $\left(\Sigma\times G,\sigma\rtimes\Psi\right)$, which will be introduced in Definition \ref{def:Koopman-M-PF} below. Combining this with
results concerning transition operators of random walks on graphs, which will be given in Section
\ref{sec:Random-Walks-Application}, we are able to give a  complete answer
to Problem \ref{mainproblem} for  potentials $\varphi$ depending only on a finite number
of coordinates (see Theorem \ref{thm:amenability-dichotomy-markov}).

Let us begin by stating the following lemma. The proof is straightforward and is thus left to the reader.
\begin{lem}
\label{lem:mu-phi-prod-counting-is-invariant}The measure \textup{\emph{$\mu_{\varphi}\times\lambda$
is $\left(\sigma\rtimes\Psi\right)$-invariant. }}
\end{lem}

Next, we define the Koopman operator (\cite{koopman31,MR1244104})
and the Perron-Frobenius operator associated to the group-extended
Markov system $\left(\Sigma\times G,\sigma\rtimes\Psi\right)$. Note
that the previous lemma ensures that these operators are well-defined. We denote by $L^{2}\left(\Sigma\times G,\mu_{\varphi}\times\lambda\right)$ the Hilbert space of real-valued functions on $\Sigma\times G$ which are square-integrable with respect to $\mu_{\varphi}\times \lambda$.
\begin{defn}
\label{def:Koopman-M-PF}The \emph{Koopman operator} $U:L^{2}\left(\Sigma\times G,\mu_{\varphi}\times\lambda\right)\rightarrow L^{2}\left(\Sigma\times G,\mu_{\varphi}\times\lambda\right)$ is given by
\[
U\left(f\right):= f\circ\left(\sigma\rtimes\Psi\right),
\]
and the \emph{Perron-Frobenius operator} $\mathcal{L}_{\varphi \circ \pi_1}:L^{2}\left(\Sigma\times G,\mu_{\varphi}\times\lambda\right)\rightarrow L^{2}\left(\Sigma\times G,\mu_{\varphi}\times\lambda\right)$
is given by
\[
\mathcal{L}_{\varphi \circ \pi_1}:=  \e^{\mathcal{P}\left(\varphi\right)}M_{h\circ\pi_{1}}\circ U^{*}\circ\left(M_{h\circ\pi_{1}}\right)^{-1},
\]
where the \emph{multiplication operator} $M_{h\circ\pi_{1}}:L^{2}\left(\Sigma\times G,\mu_{\varphi}\times\lambda\right)\rightarrow L^{2}\left(\Sigma\times G,\mu_{\varphi}\times\lambda\right)$
is given by
\[
M_{h\circ\pi_{1}}\left(f\right):=f\cdot\left(h\circ\pi_{1}\right)
\]
and $U^{*}:L^{2}\left(\Sigma\times G,\mu_{\varphi}\times\lambda\right)\rightarrow L^{2}\left(\Sigma\times G,\mu_{\varphi}\times\lambda\right)$ denotes the adjoint of $U$.
\end{defn}
The proof of the next lemma is straightforward and therefore omitted.
\begin{lem}
\label{fac:pf-fact}For the bounded linear operators  $U,\mathcal{L}_{\varphi \circ \pi_1}:L^{2}\left(\Sigma\times G,\mu_{\varphi}\times\lambda\right)\rightarrow L^{2}\left(\Sigma\times G,\mu_{\varphi}\times\lambda\right)$, 
the following properties hold.
\begin{enumerate}
\item \label{enu:UisIsometry}$U$ is an isometry, so we have that $\Vert U\Vert=\rho\left(U\right)=1$, where $\rho$ denotes the spectral radius of $U$. 
\item \label{enu:Uadjoint-is-PF}For $f\in L^{2}\left(\Sigma\times G,\mu_{\varphi}\times\lambda\right)$
and $\left(\mu_{\varphi}\times\lambda\right)$-almost every $\left(\omega,g\right)\in\Sigma\times G$
we have that \textup{
\[
\mathcal{L}_{\varphi\circ\pi_{1}}\left(f\right)\left(\omega,g\right)=\sum_{i\in I:i\omega_{1}\in\Sigma^{2}}\e{}^{\varphi\left(i\omega\right)}f\left(i\omega,g\Psi\left(i\right)^{-1}\right).
\]
}
\item \label{enu:.pf-fact-spectralradius-pressure}For the spectral radius
of $\mathcal{L}_{\varphi\circ\pi_{1}}$ we obtain that $\rho\left(\mathcal{L}_{\varphi\circ\pi_{1}}\right)=\e^{\mathcal{P}\left(\varphi\right)}$.
\end{enumerate}
\end{lem}

\global\long\def\pr#1{\1_{\left\{  \pi_{2}=#1\right\}  }}

\begin{rem*} The representation of $\mathcal{L}_{\varphi\circ\pi_{1}}$ in Lemma  \ref{fac:pf-fact} (\ref{enu:Uadjoint-is-PF}) extends Definition \ref{def:perron-frobenius}  of the Perron-Frobenius operator for Markov shifts with a finite alphabet.
\end{rem*}
The next lemma gives relationships between $\mathcal{P}\left(\varphi,\Psi^{-1}\left\{ \id\right\} \cap\Sigma^{*}\right)$
and $\mathcal{L}_{\varphi\circ\pi_{1}}$. Before stating the lemma, let us fix some notation. We write $\1_{A}$ for the
characteristic function of a set $A$ and we use $\left\{ \pi_{2}=g\right\} $
to denote the set $\pi_{2}^{-1}\left\{ g\right\}$,  for each $g\in G$. Further, let $\mathcal{B}\left(\Sigma\times G\right)$ denote the Borel $\sigma$-algebra on $\Sigma \times G$.
\begin{lem}
\label{lem:perronfrobenius-pressure}For all sets $A,B\in\mathcal{B}\left(\Sigma\times G\right)$
and for each $n\in\N$ we have that 
\[
\frac{\min h}{\max h}\mu_{\varphi}\left(A\cap\left(\sigma\rtimes\Psi\right)^{-n}\left(B\right)\right)\le\e^{-n\mathcal{P}\left(\varphi\right)}\left(\mathcal{L}_{\varphi\circ\pi_{1}}^{n}\left(\1_{A}\right),\1_{B}\right)\le\frac{\max h}{\min h}\mu_{\varphi}\left(A\cap\left(\sigma\rtimes\Psi\right)^{-n}\left(B\right)\right).
\]
Moreover, for all $g,g'\in G$ we have that
\[
\limsup_{n\rightarrow\infty}\frac{1}{n}\log\left(\mathcal{L}_{\varphi\circ\pi_{1}}^{n}\left(\pr g\right),\pr{g'}\right)=\mathcal{P}\left(\varphi,\Psi^{-1}\left\{ g^{-1}g'\right\} \cap\Sigma^{*}\right).
\]
\end{lem}
\begin{proof}
For the first assertion, observe that by the definition of $\mathcal{L}_{\varphi\circ\pi_{1}}$ we have that
\begin{eqnarray*}
\left(\mathcal{L}_{\varphi\circ\pi_{1}}^{n}\left(\1_{A}\right),\1_{B}\right) & = & \e^{n\mathcal{P}\left(\varphi\right)}\left(M_{h\circ\pi_{1}}\circ\left(U^{*}\right)^{n}\circ\left(M_{h\circ\pi_{1}}\right)^{-1}\left(\1_{A}\right),\1_{B}\right)\\
 & = & \e^{n\mathcal{P}\left(\varphi\right)}\left(\left(M_{h\circ\pi_{1}}\right)^{-1}\left(\1_{A}\right),\left(M_{h\circ\pi_{1}}\left(\1_{B}\right)\right)\circ\left(\sigma\rtimes\Psi\right)^{n}\right).
\end{eqnarray*}
Since the continuous function $h:\Sigma\rightarrow\R^{+}$ is bounded
away from zero and infinity on the compact set $\Sigma$, the first
assertion follows.

The second assertion follows from the first, if we set $A:=\left\{ \pi_{2}=g\right\} $
and $B:=\left\{ \pi_{2}=g'\right\} $ and use the Gibbs property (\ref{eq:gibbs-equation})
of $\mu_{\varphi}$.
\end{proof}
As an immediate consequence of the previous lemma, we obtain the following
upper bound for $\mathcal{P}\left(\varphi,\Psi^{-1}\left\{ g^{-1}g'\right\} \cap\Sigma^{*}\right)$
in terms of the spectral radius of $\mathcal{L}_{\varphi\circ\pi_{1}}$.
\begin{cor}
\label{cor:upperboundviaspectralradius}Let $V\subset L^{2}\left(\Sigma\times G,\mu_{\varphi}\times\lambda\right)$
be a closed $\mathcal{L}_{\varphi\circ\pi_{1}}$-invariant linear
subspace such that $\pr g, \pr{g'}\in V$, for some $g,g'\in G$. We then have that
\[
\mathcal{P}\left(\varphi,\Psi^{-1}\left\{ g^{-1}g'\right\} \cap\Sigma^{*}\right)\le\log\rho\left(\mathcal{L}_{\varphi\circ\pi_{1}}\big|_{V}\right).
\]
\end{cor}
\begin{proof}
By the Cauchy-Schwarz inequality and Gelfand's formula (\cite[Theorem 10.13]{MR0365062}) for the spectral
radius,  we have that
\[
\limsup_{n\rightarrow\infty}\frac{1}{n}\log\left(\mathcal{L}_{\varphi\circ\pi_{1}}^{n}\big|_{V}\left(\pr g\right),\pr{g'}\right)\le\limsup_{n\rightarrow\infty}\frac{1}{n}\log\Vert\mathcal{L}_{\varphi\circ\pi_{1}}^{n}\big|_{V}\Vert=\log\rho\left(\mathcal{L}_{\varphi\circ\pi_{1}}\big|_{V}\right).
\]
Combining the above inequality with the second assertion of  Lemma \ref{lem:perronfrobenius-pressure} completes
the proof.
\end{proof}
Recall that for a closed linear subspace $V\subset L^{2}\left(\Sigma\times G,\mu_{\varphi}\times\lambda\right)$,
a bounded linear operator $T:V\rightarrow V$ is called {\em positive} if
$T\left(V^{+}\right)\subset V^{+}$, where the positive cone $V^{+}$
is defined by $V^{+}:=\left\{ f\in V:f\ge0\right\} $.

The following lemma will be crucial in order to obtain equality in the inequality  stated in Corollary \ref{cor:upperboundviaspectralradius}.
The lemma extends a result of Gerl (see \cite{MR938257} and also
\cite[Lemma 10.1]{MR1743100}).
\begin{lem}
\label{lem:lowerbound-selfadjoint}Let $V$ be a closed linear subspace
of\textup{ $L^{2}\left(\Sigma\times G,\mu_{\varphi}\times\lambda\right)$
}\textup{\emph{such that \linebreak}}\textup{ $\left\{ \1_{\left\{ \pi_{2}=g\right\} }:g\in G\right\} \subset V$.
}\textup{\emph{Let $T:V\rightarrow V$ be a self-adjoint bounded linear
operator on $V$, which is positive and which satisfies}} $\ker\left(T\right)\cap V^{+}=\left\{ 0\right\} $.
We then have that
\[
\sup_{g,g'\in G}\left\{ \limsup_{n\rightarrow\infty}\left|\left(T^{n}\left(\pr g\right),\pr{g'}\right)\right|^{1/n}\right\} =\Vert T\Vert=\rho\left(T\right).
\]
\end{lem}
\begin{proof}
Since $T$ is self-adjoint, it follows that $\Vert T\Vert=\rho\left(T\right)$.
As in the proof of Corollary \ref{cor:upperboundviaspectralradius},
one immediately verifies that
\[
\sup_{g,g'\in G}\left\{ \limsup_{n\rightarrow\infty}\left|\left(T^{n}\left(\pr g\right),\pr{g'}\right)\right|^{1/n}\right\} \le\rho\left(T\right).
\]
Let us first give an outline for the proof of the opposite inequality.
We will first prove that for all $f\in V^{+}$ with $f\neq0$, the sequence $\left(\left(T^{n+1}f,T^{n+1}f\right)/\left(T^{n}f,T^{n}f\right)\right)_{n\in\N_{0}}$, 
is non-decreasing. This will then imply that the following limits
exist and are equal:
\begin{equation}
\lim_{n\rightarrow\infty}\frac{\left(T^{n+1}f,T^{n+1}f\right)}{\left(T^{n}f,T^{n}f\right)}=\lim_{n\rightarrow\infty}\left(T^{n}f,T^{n}f\right)^{1/n}.\label{eq:lowerboundselfadjoint-1a}
\end{equation}
From this we obtain for every $f\in V^{+}$ with $f\neq0$ that
\begin{equation}
\frac{\left(Tf,Tf\right)}{\left(f,f\right)}\le\lim_{n\rightarrow\infty}\left(T^{n}f,T^{n}f\right)^{1/n}.\label{eq:lowerboundselfadjoint-1aa}
\end{equation}
Subsequently, we make use of the fact that
\[
D':=\left\{ f\in L^{2}\left(\Sigma\times G,\mu_{\varphi}\times\lambda\right)\cap L^{\infty}\left(\Sigma\times G,\mu_{\varphi}\times\lambda\right):\,\, f\big|_{\left\{ \pi_{2}=g\right\} }=0\textrm{ for almost every }g\in G\right\}
\]
is dense in $L^{2}\left(\Sigma\times G,\mu_{\varphi}\times\lambda\right)$
and hence, $D:=D'\cap V$ is dense in $V$. For $f\in D$ we show
that
\[
\lim_{n\rightarrow\infty}\left(T^{n}f,T^{n}f\right)^{1/n}\le\sup_{g,g'\in G}\left\{ \limsup_{n\rightarrow\infty}\left|\left(T^{2n}\left(\pr g\right),\pr{g'}\right)\right|^{1/n}\right\} .
\]
Combining this with (\ref{eq:lowerboundselfadjoint-1aa}) applied
to $\left|f\right|$,  we conclude for $f\in D$ with $f\neq0$ that

\begin{equation}
\frac{\left(Tf,Tf\right)}{\left(f,f\right)}\le\frac{\left(T\left|f\right|,T\left|f\right|\right)}{\left(\left|f\right|,\left|f\right|\right)}\le\sup_{g,g'\in G}\left\{ \limsup_{n\rightarrow\infty}\left|\left(T^{2n}\left(\pr g\right),\pr{g'}\right)\right|^{1/n}\right\} .\label{eq:lowerboundselfadjoint-1}
\end{equation}
Since $D$ is dense in $V$, there exists a sequence $(f_n)_{n\in \N}\in D^{\N}$  such that  $\lim_n(Tf_n,Tf_n)=\Vert T \Vert$ and $(f_n,f_n)=1$, for each $n\in \N$. Combining this observation with the estimate in (\ref{eq:lowerboundselfadjoint-1}), we conclude that  $\Vert T\Vert\le\sup_{g,g'\in G}\left\{ \limsup_{n}\left|\left(T^{2n}\left(\pr g\right),\pr{g'}\right)\right|^{1/2n}\right\} $.

Let us now turn to the details. We first verify that for every $f\in V^{+}$
with $f\neq0$, the sequence $\left(a_{n}\right)_{n\in\N_{0}}$ of
positive real numbers,  given for $n\in\N_{0}$ by $a_{n}:=\left(T^{n+1}f,T^{n+1}f\right)/\left(T^{n}f,T^{n}f\right)$  is non-decreasing. Using that $T$ is self-adjoint
and applying the Cauchy-Schwarz inequality, we have for $n\in\N_{0}$
that
\begin{eqnarray}
\left(T^{n+1}f,T^{n+1}f\right)^{2} & = & \left(T^{n}f,T^{n+2}f\right)^{2}\le\left(T^{n}f,T^{n}f\right)\left(T^{n+2}f,T^{n+2}f\right).\label{eq:monotony-spectral-radius}
\end{eqnarray}
Since $\left(T^{n}f,T^{n}f\right)\neq0$ for all $n\in\N_{0}$ by
our hypothesis, we can multiply both sides of (\ref{eq:monotony-spectral-radius})
by $\left(T^{n+1}f,T^{n+1}f\right)^{-1}\left(T^{n+2}f,T^{n+2}f\right)^{-1}$,
which proves that $\left(a_{n}\right)_{n\in\N_{0}}$ is non-decreasing.
Hence, we have that $\lim_{n\rightarrow \infty}a_{n}\in\R^{+}\cup\left\{ \infty\right\} $
exists. Observing that $\log\left(T^{n}f,T^{n}f\right)$ is equal
to the telescoping sum $\log\left(f,f\right)+\sum_{j=0}^{n-1}\log a_{j}$
and using that $\lim_{n\rightarrow \infty}\log\left(a_{n}\right)$ is equal to its Ces\`{a}ro
mean, we deduce that
\[
\lim_{n\rightarrow\infty}\frac{1}{n}\log\left(T^{n}f,T^{n}f\right)=\lim_{n\rightarrow\infty}\frac{1}{n}\log\left(f,f\right)+\lim_{n\rightarrow\infty}\frac{1}{n}\sum_{j=0}^{n-1}\log a_{j}=\lim_{n\rightarrow\infty}\log a_{n},
\]
which proves (\ref{eq:lowerboundselfadjoint-1a}). Since $\left(T^{n}f,T^{n}f\right)^{1/n}\le\Vert T\Vert^{2}\max\left\{ \Vert f\Vert_{2}^2,1\right\} $,
for all $n\in\N$, we have that the limits in (\ref{eq:lowerboundselfadjoint-1a})
are both  finite.

It remains to prove that (\ref{eq:lowerboundselfadjoint-1}) holds
for every $f\in D$ with $f\neq0$. By definition of $D$,  there exists
a finite set $G_{0}\subset G$ such that $f=\sum_{g\in G_{0}}f\pr g$.
Since $T$ is positive and self-adjoint, we conclude that
\begin{eqnarray*}
\left(T^{n}f,T^{n}f\right) & \le & \left(T^{n}\left|f\right|,T^{n}\left|f\right|\right)=\left(T^{2n}\left|f\right|,\left|f\right|\right)=\sum_{g,g'\in G_{0}}\left(T^{2n}\left|f\pr g\right|,\left|f\pr{g'}\right|\right)\\
 & \le & \sum_{g,g'\in G_0}\Vert f\Vert_{L^{\infty}\left(\Sigma\times G,\mu_{\varphi}\times\lambda\right)}\left(T^{2n}\pr g,\pr{g'}\right).
\end{eqnarray*}
Finally, raising both sides of the previous inequality to the power
$1/n$ and let $n$ tend to infinity gives
\[
\lim_{n\rightarrow\infty}\left(T^{n}f,T^{n}f\right)^{1/n}\le\max_{g,g'\in G_{0}}\limsup_{n\rightarrow\infty}\left|\left(T^{2n}\pr g,\pr{g'}\right)\right|^{1/n},
\]
and the estimate in (\ref{eq:lowerboundselfadjoint-1}) follows. The
proof is complete.
\end{proof}
Regarding the requirements of the previous proposition, we prove the
following for $\mathcal{L}_{\varphi\circ\pi_{1}}$.
\begin{lem}
\label{fac:positivity-injectivitiy-of-adjoint-of-pf}Let $V$ be a
closed $\mathcal{L}_{\varphi\circ\pi_{1}}$-invariant linear subspace
of $L^{2}\left(\Sigma\times G,\mu_{\varphi}\times\lambda\right)$
and suppose that $\mathcal{L}_{\varphi}\1=\1$. Then,  $\mathcal{L}_{\varphi\circ\pi_{1}}\big|_{V}$
is a positive operator for which $\ker\left(\mathcal{L}_{\varphi\circ\pi_{1}}\big|_{V}\right)\cap V^{+}=\left\{ 0\right\} .$ Further, if $\left\{ f^{-}:f\in V\right\} \subset V$ then $\left(\mathcal{L}_{\varphi\circ\pi_{1}}\big|_{V}\right)^{*}$
is a positive operator and if there exists $g\in V$ with $g>0$,
then we have  that $\ker\left(\left(\mathcal{L}_{\varphi\circ\pi_{1}}\big|_{V}\right)^{*}\right)\cap V^{+}=\left\{ 0\right\} .$ \end{lem}
\begin{proof}
 Clearly, by definition of $\mathcal{L}_{\varphi\circ\pi_{1}}$,
we have that $\mathcal{L}_{\varphi\circ\pi_{1}}\big|_{V}$ is positive.
Now let $f\in\ker\left(\mathcal{L}_{\varphi\circ\pi_{1}}\big|_{V}\right)\cap V^{+}$. Since $\mu_{\varphi}$ is a fixed point of $\mathcal{L}_{\varphi}^*$, one deduces by the monotone convergence theorem  and by the definition of  $\mathcal{L}_{\varphi\circ\pi_{1}}$ that $\int fd\!\left(\mu_{\varphi}\times\lambda\right)=\int\mathcal{L}_{\varphi\circ\pi_{1}}\left(f\right)d\!\left(\mu_{\varphi}\times\lambda\right)$. Hence,  $f\in\ker\left(\mathcal{L}_{\varphi\circ\pi_{1}}\big|_{V}\right)\cap V^{+}$
implies $\int fd\!\left(\mu_{\varphi}\times\lambda\right)=0$ and so, $f=0$.

We now turn our attention to the adjoint  operator $\left(\mathcal{L}_{\varphi\circ\pi_{1}}\big|_{V}\right)^{*}$. Let $f\in V^+$.   Since  $\left\{ f^{-}:f\in V\right\} \subset V$ and using that  $\mathcal{L}_{\varphi\circ\pi_{1}}$ is positive, we obtain  that
\[
0\ge\left(\left(\mathcal{L}_{\varphi\circ\pi_{1}}\big|_{V}\right)^{*}\left(f\right),\left(\left(\mathcal{L}_{\varphi\circ\pi_{1}}\big|_{V}\right)^{*}\left(f\right)\right)^{-}\right)=\left(f,\mathcal{L}_{\varphi\circ\pi_{1}}\left(\left(\left(\mathcal{L}_{\varphi\circ\pi_{1}}\big|_{V}\right)^{*}\left(f\right)\right)^{-}\right)\right)\ge0.
\]
Thus, $0=\left(\left(\mathcal{L}_{\varphi\circ\pi_{1}}\big|_{V}\right)^{*}\left(f\right),\left(\left(\mathcal{L}_{\varphi\circ\pi_{1}}\big|_{V}\right)^{*}\left(f\right)\right)^{-}\right)=-\Vert\left(\left(\mathcal{L}_{\varphi\circ\pi_{1}}\big|_{V}\right)^{*}\left(f\right)\right)^{-}\Vert_{2}^{2}$
and so,  $\left(\mathcal{L}_{\varphi\circ\pi_{1}}\big|_{V}\right)^{*}$
is positive. Now let $f\in\ker\left(\left(\mathcal{L}_{\varphi\circ\pi_{1}}\big|_{V}\right)^{*}\right)\cap V^{+}$
be given and assume that there exists $g\in V$ with $g>0$. We then have that
\[
0=\left(\left(\mathcal{L}_{\varphi\circ\pi_{1}}\big|_{V}\right)^{*}\left(f\right),g\right)=\left(f,\mathcal{L}_{\varphi\circ\pi_{1}}\left(g\right)\right).
\]
Since $g>0$, we have $\mathcal{L}_{\varphi\circ\pi_{1}}\left(g\right)>0$,
which implies that $f=0$.  The proof is complete.
\end{proof}
It turns out that the   Perron-Frobenius operator is not self-adjoint in general. In fact, as we will see in the following remark, this operator is self-adjoint only in very special cases. Therefore, we will introduce the notion of an asymptotically self-adjoint operator in Definition \ref{def:asymptotically-sefadjoint} below. 
\begin{rem*} We observe that the requirement that $\mathcal{L}_{\varphi\circ\pi_{1}}\big|_{V}$
is self-adjoint, for some closed linear subspace $V$ of $L^{2}\left(\Sigma\times G,\mu_{\varphi}\times\lambda\right)$,
is rather restrictive. Indeed, suppose that $\mathcal{L}_{\varphi\circ\pi_{1}}\big|_{V}$ is self-adjoint
for a closed linear subspace $V$ of $L^{2}\left(\Sigma\times G,\mu_{\varphi}\times\lambda\right)$
satisfying $\left\{ \1_{\left[i\right]\times\left\{ g\right\} }:i\in I,g\in G\right\} \subset V$.
 It follows that  $ji\in\Sigma^{2}$ and $\Psi\left(i\right)=\Psi\left(j\right)^{-1}$,  for  all $i,j\in I$ such that  $ij\in\Sigma^{2}$.
In particular, we have that $\Psi\left(\Sigma^{*}\right)$ has at
most two elements. To prove  this, let $ij\in\Sigma^{2}$ be given. By the Gibbs property (\ref{eq:gibbs-equation}) of $\mu_\varphi$ we have that  $\mu_{\varphi}[ij]>0$. Setting  $C:=\frac{\max h}{\min h}\e^{-\mathcal{P}(\varphi)}$ we deduce from Lemma \ref{lem:perronfrobenius-pressure} that
\[
0<\left(\mu_{\varphi}\times\lambda\right)\left(\left(\left[i\right]\times\left\{ \id\right\} \right)\cap\left(\sigma\rtimes\Psi\right)^{-1}\left(\left[j\right]\times\left\{ \Psi(i)\right\} \right)\right)\le C  \left(\mathcal{L}_{\varphi\circ\pi_{1}}\big|_{V}\left(\1_{\left[i\right]\times\left\{ \id\right\} }\right),\1_{\left[j\right]\times\left\{ \Psi\left(i\right)\right\} }\right).
\]
Using that $\mathcal{L}_{\varphi\circ\pi_{1}}\big|_{V}$ is self-adjoint and again by Lemma \ref{lem:perronfrobenius-pressure}, we conclude that
\[
\left(\mathcal{L}_{\varphi\circ\pi_{1}}\big|_{V}\left(\1_{\left[i\right]\times\left\{ \id\right\} }\right),\1_{\left[j\right]\times\left\{ \Psi\left(i\right)\right\} }\right)\le C  \left(\mu_{\varphi}\times\lambda\right)\left(\left(\left[j\right]\times\left\{ \Psi\left(i\right)\right\} \right)\cap\left(\sigma\rtimes\Psi\right)^{-1}\left(\left[i\right]\times\left\{ \id\right\} \right)\right).
\]
Combining the previous two estimates, we conclude that $\left(\left[j\right]\times\left\{ \Psi\left(i\right)\right\} \right)\cap\left(\sigma\rtimes\Psi\right)^{-1}\left(\left[i\right]\times\left\{ \id\right\} \right)$
is nonempty, hence $ji\in\Sigma^{2}$ and $\Psi\left(i\right)\Psi\left(j\right)=\id$.
\end{rem*}

The following definition introduces a concept which is slightly weaker than self-adjointness.
\begin{defn}\label{def:asymptotically-sefadjoint}Let $V$ be a closed linear subspace of  $L^{2}\left(\Sigma\times G,\mu_{\varphi}\times\lambda\right)$ and let  $T:V\rightarrow V$ be a positive bounded linear operator. We say that $T$ is \emph{asymptotically self-adjoint} if there exist sequences
$\left(c_{m}\right)_{m\in\N}\in\left(\R^{+}\right)^{\N}$ and $\left(N_{m}\right)_{m\in\N}\in\N_0^{\N}$
with the property that $\lim_{m\to\infty}$$\left(c_{m}\right)^{1/m}=1$, $\lim_{m\to\infty}m^{-1}N_{m}=0$,  such that for all non-negative functions $f,g\in V$ and for all $n\in\N$ we have
\begin{equation}
\left(T^{n}f,g\right)\le c_{n}\sum_{i=0}^{N_{n}}\left(f,T^{n+i}g\right).
\label{eq:asymptotically-selfadjoint-condition}
\end{equation}
\end{defn}
\begin{rem*}
Note that $T$ is asymptotically  self-adjoint if and only if $T^{*}$ is
asymptotically  self-adjoint. We also remark that it clearly suffices to verify
(\ref{eq:asymptotically-selfadjoint-condition}) on a norm-dense subset
of non-negative functions in $V$.
\end{rem*}
The next proposition shows that if $\mathcal{L}_{\varphi\circ\pi_{1}}\big|_{V}$
is asymptotically self-adjoint, for some closed linear subspace $V$ of $L^{2}\left(\Sigma\times G,\mu_{\varphi}\times\lambda\right)$,
then we can relate the supremum of $\mathcal{P}\left(\varphi,\Psi^{-1}\left\{ g\right\} \cap\Sigma^{*}\right)$, 
for $g\in G$,  to  the spectral radius of $\mathcal{L}_{\varphi\circ\pi_{1}}\big|_{V}$.
The proof, which makes use of Lemma \ref{lem:lowerbound-selfadjoint}
and Lemma \ref{fac:positivity-injectivitiy-of-adjoint-of-pf}, is
inspired by \cite[Proposition 1.6]{MR2338235}.
\begin{prop}
\label{pro:lowerbound-asymptselfadjoint}Suppose that $\mathcal{L}_{\varphi}\1=\1$
and let $V\subset L^{2}\left(\Sigma\times G,\mu_{\varphi}\times\lambda\right)$
be a closed linear $\mathcal{L}_{\varphi\circ\pi_{1}}$-invariant
subspace such that $\left\{ f^{-}:f\in V\right\} \subset V$ and $\left\{ \pr g:g\in G\right\} \subset V$.
If \textup{$\mathcal{L}_{\varphi\circ\pi_{1}}\big|_{V}$ } is asymptotically
self-adjoint, then we have that 
\[
\sup_{g\in G}\left\{ \mathcal{P}\left(\varphi,\Psi^{-1}\left\{ g\right\} \cap\Sigma^{*}\right)\right\} =\log\rho\left(\mathcal{L}_{\varphi\circ\pi_{1}}\big|_{V}\right).
\]
\end{prop}
\begin{proof}
By Corollary \ref{cor:upperboundviaspectralradius}, we have $\sup_{g\in G}\left\{ \mathcal{P}\left(\varphi,\Psi^{-1}\left\{ g\right\} \cap\Sigma^{*}\right)\right\} \le\log\rho\left(\mathcal{L}_{\varphi\circ\pi_{1}}\big|_{V}\right)$.
Let us turn to the proof of the converse inequality. Using that $\Vert\left(\mathcal{L}_{\varphi\circ\pi_{1}}^{m}\big|_{V}\right)^{*}\mathcal{L}_{\varphi\circ\pi_{1}}^{m}\big|_{V}\Vert=\Vert\mathcal{L}_{\varphi\circ\pi_{1}}^{m}\big|_{V}\Vert^{2}$
for each $m\in\N$, it follows from  Gelfand's formula  (\cite[Theorem 10.13]{MR0365062})  that
\begin{equation}
\rho\left(\mathcal{L}_{\varphi\circ\pi_{1}}\big|_{V}\right)=\lim_{n\rightarrow\infty}\Vert\left(\mathcal{L}_{\varphi\circ\pi_{1}}^{n}\big|_{V}\right)^{*}\mathcal{L}_{\varphi\circ\pi_{1}}^{n}\big|_{V}\Vert^{1/(2n)}.\label{eq:lowerbound-asymptselfadjoint-1}
\end{equation}
Our next aim is to apply Lemma \ref{lem:lowerbound-selfadjoint} to the self-adjoint operator
$T_n:=\left(\mathcal{L}_{\varphi\circ\pi_{1}}^{n}\big|_{V}\right)^{*}\mathcal{L}_{\varphi\circ\pi_{1}}^{n}\big|_{V}$,
for each $n\in\N$. We have to verify that $T_n$
is positive and that $\ker(T_n)\cap V^{+}=\left\{ 0\right\} $,
for each $n\in\N$.  By Lemma \ref{fac:positivity-injectivitiy-of-adjoint-of-pf}, 
we have that $\mathcal{L}_{\varphi\circ\pi_{1}}\big|_{V}$ is positive
and $\ker\left(\mathcal{L}_{\varphi\circ\pi_{1}}\big|_{V}\right)\cap V^{+}=\left\{ 0\right\} $.
Fix some arbitrary order for the elements in $G$, say  $G=\left\{ g_{i}:i\in\N\right\} $. Using that $V$ is a closed linear subspace containing $\left\{ \pr{g_{i}}:i\in\N\right\} $,
we obtain that $g:=\sum_{j\in\N}2^{-j}\1_{\left\{ \pi_{2}=g_{j}\right\} }>0$
is an element of $V$. Since  $\left\{ f^{-}:f\in V\right\} \subset V$
by our assumptions, Lemma \ref{fac:positivity-injectivitiy-of-adjoint-of-pf}
implies that  $\left(\mathcal{L}_{\varphi\circ\pi_{1}}\big|_{V}\right)^{*}$
is positive with $\ker\left(\mathcal{L}_{\varphi\circ\pi_{1}}\big|_{V}\right)^{*}\cap V^{+}=\left\{ 0\right\} $.
Hence, for each $n\in\N$ we have that $T_n$
is positive and $\ker(T_n)\cap V^{+}=\left\{ 0\right\} $.
Consequently, it follows from  Lemma \ref{lem:lowerbound-selfadjoint} that  for each $n\in\N$ we have
\begin{eqnarray}
 &  & \Vert\left(\mathcal{L}_{\varphi\circ\pi_{1}}^{n}\big|_{V}\right)^{*}\mathcal{L}_{\varphi\circ\pi_{1}}^{n}\big|_{V}\Vert=\sup_{g,g'\in G}\left\{ \limsup_{k\rightarrow\infty}\left(\left(\left(\mathcal{L}_{\varphi\circ\pi_{1}}^{n}\big|_{V}\right)^{*}\mathcal{L}_{\varphi\circ\pi_{1}}^{n}\big|_{V}\right)^{k}\left(\pr g\right),\pr{g'}\right)^{1/k}\right\} .\label{eq:lowerbound-asymptselfadjoint-2}
\end{eqnarray}
Let $g,g'\in G$ be given. Using that $\mathcal{L}_{\varphi\circ\pi_{1}}\big|_{V}$
is asymptotically  self-adjoint, with sequences $\left(c_{m}\right)_{m\in\N}\in\R^{\N}$
and $\left(N_{m}\right)_{m\in\N}\in\N_0^{\N}$ as in Definition \ref{def:asymptotically-sefadjoint},
we estimate for all $n\in\N$ that
\begin{eqnarray*}
 &  & \limsup_{k\rightarrow\infty}\big(\left(\left(\mathcal{L}_{\varphi\circ\pi_{1}}^{n}\big|_{V}\right)^{*}\mathcal{L}_{\varphi\circ\pi_{1}}^{n}\big|_{V}\right)^{k}\left(\pr g\right),\pr{g'}\big)^{1/k}\\
 & \le & \limsup_{k\rightarrow\infty}\big(c_{n}^{k}\sum_{i_{1}=0}^{N_{n}}\sum_{i_{2}=0}^{N_{n}}\dots\sum_{i_{k}=0}^{N_{n}}\big(\mathcal{L}_{\varphi\circ\pi_{1}}^{2nk+\sum_{j=1}^{k}i_{j}}\big|_{V}\left(\pr g\right),\pr{g'}\big)\big)^{1/k}\\
 & \le & c_{n}\limsup_{k\rightarrow\infty}\big(\left(N_{n}+1\right)^{k}\max_{\left(i_{1},\dots,i_{k}\right)\in\left\{ 0,\dots,N_{n}\right\} ^{k}}\big\{\big(\mathcal{L}_{\varphi\circ\pi_{1}}^{2nk+\sum_{j=1}^{k}i_{j}}\big|_{V}\left(\pr g\right),\pr{g'}\big)\big\}\big)^{1/k}.
\end{eqnarray*}
Let $\epsilon>0$. Since we have $\limsup_{m\rightarrow\infty}\left(\mathcal{L}_{\varphi\circ\pi_{1}}^{m}\big|_{V}\left(\pr g\right),\pr{g'}\right)^{1/m}=\e^{\mathcal{P}\left(\varphi,\Psi^{-1}\left\{ g^{-1}g'\right\} \cap\Sigma^{*}\right)}$
by Lemma \ref{lem:perronfrobenius-pressure}, we obtain that
\begin{align*}
 & \limsup_{k\rightarrow\infty}\big(\max_{\left(i_{1},\dots,i_{k}\right)\in\left\{ 0,\dots,N_{n}\right\} ^{k}}\big\{\big(\mathcal{L}_{\varphi\circ\pi_{1}}^{2nk+\sum_{j=1}^{k}i_{j}}\big|_{V}\left(\pr g\right),\pr{g'}\big)\big\}\big)^{1/k}\le\\
 & \limsup_{k\rightarrow\infty}\big(\max_{\left(i_{1},\dots,i_{k}\right)\in\left\{ 0,\dots,N_{n}\right\} ^{k}}\max\big\{\e^{\left(2nk+kN_{n}\right)\left(\mathcal{P}\left(\varphi,\Psi^{-1}\left\{ g^{-1}g'\right\} \cap\Sigma^{*}\right)+\epsilon\right)},\e^{2nk\left(\mathcal{P}\left(\varphi,\Psi^{-1}\left\{ g^{-1}g'\right\} \cap\Sigma^{*}\right)+\epsilon\right)}\big\}\big)^{1/k}.
\end{align*}
Since $\epsilon>0$ was chosen to be arbitrary, our previous estimates
imply that for each $n\in\N$ and for all $g,g'\in G$ we have
\begin{eqnarray}
 &  & \limsup_{k\rightarrow\infty}\big(\left(\left(\mathcal{L}_{\varphi\circ\pi_{1}}^{n}\big|_{V}\right)^{*}\mathcal{L}_{\varphi\circ\pi_{1}}^{n}\big|_{V}\right)^{k}\left(\pr g\right),\pr{g'}\big)^{1/k}\label{eq:lower-bound-estimate}\\
 & \le & c_{n}\left(N_{n}+1\right)\max\left\{ \e^{\left(2n+N_{n}\right)\mathcal{P}\left(\varphi,\Psi^{-1}\left\{ g^{-1}g'\right\} \cap\Sigma^{*}\right)},\e^{2n\mathcal{P}\left(\varphi,\Psi^{-1}\left\{ g^{-1}g'\right\} \cap\Sigma^{*}\right)}\right\} .\nonumber
\end{eqnarray}
Combining (\ref{eq:lowerbound-asymptselfadjoint-1}), (\ref{eq:lowerbound-asymptselfadjoint-2})
and (\ref{eq:lower-bound-estimate}), we obtain that
\begin{eqnarray*}
 &  & \rho\left(\mathcal{L}_{\varphi\circ\pi_{1}}\big|_{V}\right)=\lim_{n\rightarrow\infty}\Vert\left(\mathcal{L}_{\varphi\circ\pi_{1}}^{n}\big|_{V}\right)^{*}\mathcal{L}_{\varphi\circ\pi_{1}}^{n}\big|_{V}\Vert^{1/(2n)}\\
 & \le & \limsup_{n\rightarrow\infty}\big(c_{n}\left(N_{n}+1\right)\sup_{g,g'\in G}\max\big\{\e^{\left(2n+N_{n}\right)\mathcal{P}\left(\varphi,\Psi^{-1}\left\{ g^{-1}g'\right\} \cap\Sigma^{*}\right)},\e^{2n\mathcal{P}\left(\varphi,\Psi^{-1}\left\{ g^{-1}g'\right\} \cap\Sigma^{*}\right)}\big\}\big)^{1/(2n)}\\
 & \le & \lim_{n\rightarrow\infty}\left(c_{n}\left(N_{n}+1\right)\right)^{1/(2n)}\sup_{g\in G}\max\big\{\e^{\left(1+N_{n}/(2n)\right)\mathcal{P}\left(\varphi,\Psi^{-1}\left\{ g\right\} \cap\Sigma^{*}\right)},\e^{\mathcal{P}\left(\varphi,\Psi^{-1}\left\{ g\right\} \cap\Sigma^{*}\right)}\big\}.
\end{eqnarray*}
Since $\lim_{n\rightarrow \infty}\left(c_{n}\right)^{1/n}=1$ and $\lim_{n\rightarrow \infty}n^{-1}N_{n}=0$,
the proof is complete.
\end{proof}
In the following definition, we introduce certain important closed linear subspaces of the space $L^{2}\left(\Sigma\times G,\mu_{\varphi}\times\lambda\right)$. 
\begin{defn}
\label{def:C-k-measurable-subspaces} For $j\in\N_{0}$,  let $V_{j}\subset L^{2}\left(\Sigma\times G,\mu_{\varphi}\times\lambda\right)$
denote the subspace consisting of all $f\in L^{2}\left(\Sigma\times G,\mu_{\varphi}\times\lambda\right)$
which possess a $\mathcal{C}(j)\otimes\mathcal{B}\left(G\right)$-measurable
representative in $L^{2}\left(\Sigma\times G,\mu_{\varphi}\times\lambda\right)$, where $\mathcal{C}(j)\otimes\mathcal{B}\left(G\right)$ denotes
the product $\sigma$-algebra of $\mathcal{C}(j)$ and the Borel $\sigma$-algebra  $\mathcal{B}\left(G\right)$ on $G$.
\end{defn}
Note that $V_{j}$ is a Hilbert space for each $j\in\N_{0}$. The next
lemma gives an invariance property for $V_{j}$ with respect to $\mathcal{L}_{\varphi\circ\pi_{1}}$
for $\mathcal{C}(k)$-measurable potentials $\varphi$.
\begin{lem}
\label{lem:vp_invariantsubspaces}Let $\varphi:\Sigma\rightarrow\R$
be $\mathcal{C}(k)$-measurable for some $k\in\N_{0}$. Then $V_{j}$
is $\mathcal{L}_{\varphi\circ\pi_{1}}$-invariant for each $j\in\N$
with $j\ge k-1$. Moreover, for all $j\in\N_{0}$ we have that $U\left(V_{j}\right)\subset V_{j+1}$. \end{lem}
\begin{proof}
If $f$ is $\mathcal{C}\left(j\right)$-measurable, $j\in \N_0$,  then it follows from Lemma \ref{fac:pf-fact}
(\ref{enu:Uadjoint-is-PF})  that $\mathcal{L}_{\varphi\circ\pi_{1}}\left(f\right)$
is given by
\[
\mathcal{L}_{\varphi\circ\pi_{1}}\left(f\right)\left(\omega,g\right)=\sum_{i\in I:i\omega_{1}\in\Sigma^{2}}\e^{\varphi\left(i\omega\right)}f\left(i\omega,g\Psi\left(i\right)^{-1}\right).
\]
Note that the right-hand side of the previous equation depends only  on
$g\in G$ and on the elements $\omega_{1},\dots,\omega_{\max\left\{ k-1,j-1,1\right\} }\in I$.
Consequently, for $j\in\N$ with $j\ge k-1$, we have that $V_{j}$
is $\mathcal{L}_{\varphi\circ\pi_{1}}$-invariant.

The remaining assertion follows immediately from the definition of
$U$.
\end{proof}
We need the following notion of symmetry.
\begin{defn}\label{def:symmetry} We say that $\varphi:\Sigma\rightarrow\R$
is \emph{asymptotically symmetric with respect to $\Psi$ }if there exist
sequences $\left(c_{m}\right)_{m\in\N}\in\left(\R^{+}\right)^{\N}$
and $\left(N_{m}\right)_{m\in\N}\in\N_0^{\N}$ with the property that
$\lim_{m}\left(c_{m}\right)^{1/m}=1$, $\lim_{m}m^{-1}N_{m}=0$ and
such that for each $g\in G$ and for all $n\in\N$ we have
\begin{equation}
\sum_{\omega\in\Sigma^{n}:\Psi\left(\omega\right)=g}\e^{S_{\omega}\varphi}\le c_{n}\sum_{i=0}^{N_{n}}\sum_{\tau\in\Sigma^{n+i}:\Psi\left(\tau\right)=g^{-1}}\e^{S_{\tau}\varphi}.\label{eq:symmetry-definition}
\end{equation}
\end{defn}
\begin{rem}
\label{rem:symmetry-invariant-coboundary}If $\varphi$ is asymptotically
symmetric with respect to $\Psi$,  then it is straightforward to verify
that, for each $\psi:\Sigma\rightarrow\R^{+}$ H\"older continuous and
$c\in\R$, we have that  also $\varphi+\log\psi-\log\psi\circ\sigma+c$ is asymptotically
symmetric with respect to $\Psi$. Using the Gibbs property (\ref{eq:gibbs-equation})
of $\mu_{\varphi}$, an equivalent way to state that $\varphi$ is
asymptotically symmetric with respect to $\Psi$ is the following: there exist
sequences $\left(c_{m}'\right)_{m\in\N}\in\left(\R^{+}\right)^{\N}$
and $\left(N_{m}'\right)_{m\in\N}\in\N_0^{\N}$ with the property that
$\lim_{m}\left(c_{m}'\right)^{1/m}=1$, $\lim_{m}m^{-1}N_{m}'=0$
and such that for each $g\in G$ and for all $n\in\N$ we have
\[
\sum_{\omega\in\Sigma^{n}:\Psi\left(\omega\right)=g}\mu_{\varphi}\left(\left[\omega\right]\right)\le c_{n}'\sum_{i=0}^{N_{n}'}\sum_{\tau\in\Sigma^{n+i}:\Psi\left(\tau\right)=g^{-1}}\mu_{\varphi}\left(\left[\tau\right]\right).
\]

\end{rem}
Next lemma gives a necessary and sufficient condition for $\mathcal{L}_{\varphi\circ\pi_{1}}\big|_{V_{j}}$
to be asymptotically self-adjoint.
\begin{lem}
\label{lem:weaklysymmetry-is-weakselfadjoint}Let $\varphi:\Sigma\rightarrow\R$
be $\mathcal{C}(k)$-measurable,  for some $k\in\N_{0}$. For each
$j\in\N$ with $j\ge k-1$, we then have   that $\varphi$ is asymptotically symmetric with respect to $\Psi$ if and only if  $\mathcal{L}_{\varphi\circ\pi_{1}}\big|_{V_{j}}$ is asymptotically self-adjoint.
\end{lem}
\begin{proof}
We first observe that by  Lemma \ref{lem:perronfrobenius-pressure} and by the Gibbs property (\ref{eq:gibbs-equation}) of $\mu_{\varphi}$, there exists $K>0$ such that for all $n\in \N$ and for all $g,g' \in G$ we have
\begin{equation}
K^{-1}\le\frac{\left(\mathcal{L}_{\varphi\circ\pi_{1}}^{n}\left(\1_{\Sigma\times\left\{ g\right\} }\right),\1_{\Sigma\times\left\{ g'\right\} }\right)}{\sum_{\tau\in\Sigma^{n}:g\Psi\left(\tau\right)=g'}\e^{S_{\tau}\varphi}}\le K,\label{eq:perronfrobenius-sums}
\end{equation}
unless nominator and denominator in (\ref{eq:perronfrobenius-sums}) are both equal to zero.
From (\ref{eq:perronfrobenius-sums})  we obtain that $\varphi$ is asymptotically symmetric with respect to $\Psi$ if and only if there exist sequences  $\left(c_{m}\right)\in\left(\R^{+}\right)^{\N}$ and $\left(N_{m}\right)\in\N_0^{\N}$,  as in Definition \ref{def:symmetry},   such that for all $n\in\N$ and $g,g'\in G$ we have
\begin{equation}
\left(\mathcal{L}_{\varphi\circ\pi_{1}}^{n}\left(\1_{\Sigma\times\left\{ g\right\} }\right),\1_{\Sigma\times\left\{ g'\right\} }\right)\le c_{n}\left(\1_{\Sigma\times\left\{ g\right\} },\sum_{i=0}^{N_{n}}\mathcal{L}_{\varphi\circ\pi_{1}}^{n+i}\left(\1_{\Sigma\times\left\{ g'\right\} }\right)\right).\label{eq:symmetry-via-perronfrobenius-sums}
\end{equation}
Since $V_0 \subset V_j$ for each $j\in \N_0$, we obtain that if $\mathcal{L}_{\varphi\circ\pi_{1}}\big|_{V_{j}}$ is asymptotically self-adjoint, then $\varphi$ is asymptotically symmetric with respect to $\Psi$.

For the opposite implication, let $j\in \N$, $j\ge k-1$ and suppose that $\varphi$ is asymptotically symmetric with respect to $\Psi$. By Lemma \ref{lem:vp_invariantsubspaces},  we have that  $V_{j}$ is $\mathcal{L}_{\varphi\circ\pi_{1}}$-invariant. Next, we note that since $\varphi$ is asymptotically symmetric with respect to $\Psi$, we have that,  for each $\omega \in \Sigma^j$,  there exists $\kappa(\omega)\in \Sigma^*$ such that $\Psi(\omega)\Psi(\kappa(\omega))=\id$. Combining this with item (3) of our standing assumptions, we conclude that for all $\omega,\omega'\in \Sigma^j$ there exists a finite-to-one map which maps $\tau\in\Sigma^*$ to an element $\omega\gamma_1 \kappa(\omega)\gamma_2 \tau \gamma_3 \omega' \in \Sigma^*$,  where $\Psi(\gamma_i)=\id$ and  $\gamma_i$ depends only on the preceding and successive symbol, for each $i\in \{ 1,2,3\}$. Hence, in view of Lemma \ref{lem:perronfrobenius-pressure} and the Gibbs property (\ref{eq:gibbs-equation}) of $\mu_{\varphi}$, and by using that $\Sigma^j$ is finite, we conclude that there exist $N\in\N$ and $C>0$ (depending on $j$) such that for all $n\in\N$, $g,g'\in G$ and for all $\omega, \omega' \in \Sigma^j$ we have
\begin{equation}
\left(\mathcal{L}_{\varphi\circ\pi_{1}}^{n}\left(\1_{\Sigma\times\left\{ g\right\} }\right),\1_{\Sigma\times\left\{ g'\right\} }\right)\le C \sum_{r=0}^N{\left(\mathcal{L}_{\varphi\circ\pi_{1}}^{n+r}\left(\1_{[\omega]\times\left\{ g\right\} }\right),\1_{[\omega']\times\left\{ g'\right\} }\right)}.
\label{eq:symmetry-via-perronfrobenius-sums-2}
\end{equation}
By first using (\ref{eq:symmetry-via-perronfrobenius-sums}) and then (\ref{eq:symmetry-via-perronfrobenius-sums-2}),  we deduce that for all $n\in\N$, $g,g' \in G$ and for all $\omega,\omega'\in \Sigma^j$, 
\begin{align*}
\left(\mathcal{L}_{\varphi\circ\pi_{1}}^{n}\left(\1_{\Sigma\times\left\{ g\right\} }\right),\1_{\Sigma\times\left\{ g'\right\} }\right) & \le  c_{n}\left(\1_{\Sigma\times\left\{ g\right\} },\sum_{i=0}^{N_{n}}\mathcal{L}_{\varphi\circ\pi_{1}}^{n+i}\left(\1_{\Sigma\times\left\{ g'\right\} }\right)\right) \\
 & \le c_{n} C\sum_{i=0}^{N_{n}} \sum_{r=0}^N{\left( \1_{[\omega]\times\left\{ g\right\} },\mathcal{L}_{\varphi\circ\pi_{1}}^{n+i+r}\left(\1_{[\omega']\times\left\{ g'\right\} }\right) \right)} \\
 & \le c_{n} CN \sum_{i=0}^{N_{n}+N}  {\left( \1_{[\omega]\times\left\{ g\right\} },\mathcal{L}_{\varphi\circ\pi_{1}}^{n+i}\left(\1_{[\omega']\times\left\{ g'\right\} }\right) \right)}.
\end{align*}
Since $\left(\mathcal{L}_{\varphi\circ\pi_{1}}^{n}\left(\1_{[\omega]\times\left\{ g\right\} }\right),\1_{[\omega']\times\left\{ g'\right\} }\right) \le \left(\mathcal{L}_{\varphi\circ\pi_{1}}^{n}\left(\1_{\Sigma\times\left\{ g\right\} }\right),\1_{\Sigma\times\left\{ g'\right\} }\right)$  for all $\omega, \omega' \in \Sigma^j$, it follows that  $\mathcal{L}_{\varphi\circ\pi_{1}}\big|_{V_{j}}$ is asymptotically self-adjoint with respect to the sequences $\left(c_{m}'\right)\in\left(\R^{+}\right)^{\N}$ and $\left(N_{m}'\right)\in\N_0^{\N}$,   which are  given by  $c_{m}':=c_{m}CN$ and $N_{m}':=N_{m}+N$.  The proof is complete.
\end{proof}
The following corollary is a consequence of Proposition \ref{pro:lowerbound-asymptselfadjoint} 
and clarifies the relation between $\sup_{g\in G}\left\{ \mathcal{P}\left(\varphi,\Psi^{-1}\left\{ g\right\} \cap\Sigma^{*}\right)\right\} $
and the spectral radius of $\mathcal{L}_{\varphi\circ\pi_{1}}\big|_{V_{j}}$
provided that $\varphi$ is asymptotically symmetric with respect to $\Psi$.
\begin{cor}
\label{cor:pressure-is-spectralradius}Let $\varphi:\Sigma\rightarrow\R$
be $\mathcal{C}(k)$-measurable, for some $k\in\N_{0}$, and suppose
that $\varphi$ is asymptotically symmetric with respect to $\Psi$. For 
 each $j\in\N$ with $j\ge k-1$,  we then have that
\[
\sup_{g\in G}\left\{ \mathcal{P}\left(\varphi,\Psi^{-1}\left\{ g\right\} \cap\Sigma^{*}\right)\right\} =\log\rho\left(\mathcal{L}_{\varphi\circ\pi_{1}}\big|_{V_{j}}\right).
\]
\end{cor}
\begin{proof}
Fix $j\in\N$ with $j\ge k-1$. By Lemma \ref{lem:vp_invariantsubspaces},
we then have that $V_{j}$ is $\mathcal{L}_{\varphi\circ\pi_{1}}$-invariant.
Let us first verify that without loss of generality we may assume
that $\mathcal{L}_{\varphi}\1=\1$. Otherwise, by Theorem \ref{thm:perron-frobenius-thm-urbanski},
there exists a $\mathcal{C}\!\left(\max\left\{ k-1,1\right\} \right)$-measurable
function $h:\Sigma\rightarrow\R^{+}$ with $\mathcal{L}_{\varphi}\left(h\right)=\e^{\mathcal{P}\left(\varphi\right)}h$.
For $\tilde{\varphi}:=\varphi+\log h-\log h\circ\sigma-\mathcal{P}\left(\varphi\right)$,
we then have that $\mathcal{L}_{\tilde{\varphi}}\1=\1$, $\mathcal{P}\left(\tilde{\varphi}\right)=0$
and,  for each $g\in G$,
\[
\mathcal{P}\left(\tilde{\varphi},\Psi^{-1}\left\{ g\right\} \cap\Sigma^{*}\right)=\mathcal{P}\left(\varphi,\Psi^{-1}\left\{ g\right\} \cap\Sigma^{*}\right)-\mathcal{P}\left(\varphi\right).
\]
Further,  we have that 
$\tilde{\varphi}$ is asymptotically  symmetric with respect to $\Psi$,  by Remark \ref{rem:symmetry-invariant-coboundary}.  It
remains to show that $V_{j}$ is $\mathcal{L}_{\tilde{\varphi}\circ\pi_{1}}$-invariant
and that
\begin{equation}
\log\rho\left(\mathcal{L}_{\tilde{\varphi}\circ\pi_{1}}\big|_{V_{j}}\right)=\log\rho\left(\mathcal{L}_{\varphi\circ\pi_{1}}\big|_{V_{j}}\right)-\mathcal{P}\left(\varphi\right).\label{eq:spectrum-under-cocycle}
\end{equation}
Since $h$ is $\mathcal{C}\!\left(\max\left\{ k-1,1\right\} \right)$-measurable,
we have that $V_{j}$ is $M_{h\circ\pi_{1}}$-invariant and,   by the definition of the Perron-Frobenius operator, we obtain that
\[
\mathcal{L}_{\tilde{\varphi}\circ\pi_{1}}\big|_{V_{j}}=\e^{-\mathcal{P}\left(\varphi\right)}\left(M_{h\circ\pi_{1}}\big|_{V_{j}}\right)^{-1}\circ\left(\mathcal{L}_{\varphi\circ\pi_{1}}\big|_{V_{j}}\right)\circ\left(M_{h\circ\pi_{1}}\big|_{V_{j}}\right).
\]
We conclude that $V_{j}$ is $\mathcal{L}_{\tilde{\varphi}\circ\pi_{1}}$-invariant
and that $\mathcal{L}_{\tilde{\varphi}\circ\pi_{1}}\big|_{V_{j}}$ and $\e^{-\mathcal{P}\left(\varphi\right)}\mathcal{L}_{\varphi\circ\pi_{1}}\big|_{V_{j}}$
have the same spectrum. The latter fact gives the equality in  (\ref{eq:spectrum-under-cocycle}).  Hence, we may assume without loss of generality that $\mathcal{L}_{\varphi}\1=\1$.

By Lemma \ref{lem:weaklysymmetry-is-weakselfadjoint},  we have that
$\mathcal{L}_{\varphi\circ\pi_{1}}\big|_{V_{j}}$
is asymptotically self-adjoint. Since the closed linear subspace $V_{j}\subset L^{2}\left(\Sigma\times G,\mu_{\varphi}\times\lambda\right)$
satisfies $\left\{ f^{-}:f\in V_{j}\right\} \subset V_{j}$  and $\left\{ \1_{\left\{ \pi_{2}=g\right\} }:g\in G\right\} \subset V_{j}$,
the assertion of the corollary follows from Proposition \ref{pro:lowerbound-asymptselfadjoint}.
\end{proof}

\begin{rem*}
Note that, in particular, under the assumptions of the previous corollary
we have that $\rho\left(\mathcal{L}_{\varphi\circ\pi_{1}}\big|_{V_{j}}\right)$
is  independent of $j\in\N$ for all $j\ge k-1$.
\end{rem*}

\subsection{Random Walks on Graphs and Amenability\label{sec:Random-Walks-Application}}

In this section we relate the Perron-Frobenius operator to  the transition
operator of a certain random walk on a graph. We start by introducing
the following graphs.
\begin{defn}
\label{def:k-cylinder-graphs}For each $j\in\N_{0}$,  the \emph{$j$-step
graph of $\left(\Sigma\times G,\sigma\rtimes\Psi\right)$} consists
of the vertex set $\Sigma^{j}\times G$ where two vertices $\left(\omega,g\right),\left(\omega',g'\right)\in\Sigma^{j}\times G$
are connected by an edge in $X_{j}$ if and only if
\[
\left(\sigma\rtimes\Psi\right)^{-1}\left(\left[\omega\right]\times\left\{ g\right\} \right)\cap\left(\left[\omega'\right]\times\left\{ g'\right\} \right)\neq\emptyset\,\,\,\mbox{or}\,\,\left(\sigma\rtimes\Psi\right)^{-1}\left(\left[\omega'\right]\times\left\{ g'\right\} \right)\cap\left(\left[\omega\right]\times\left\{ g\right\} \right)\neq\emptyset.
\]
We use $X_{j}\left(\Sigma\times G,\sigma\rtimes\Psi\right)$ or simply
$X_{j}$ to denote this graph.
\end{defn}
Provided that $\Psi\left(\Sigma^{*}\right)=G$, we have that each
$j$-step graph of \emph{$\left(\Sigma\times G,\sigma\rtimes\Psi\right)$}
is connected. Next lemma shows that each of these graphs is roughly
isometric to the Cayley graph of $G$ with respect to $\Psi\left(I\right)\cup\Psi\left(I\right)^{-1}$
denoted by $X\!\left(G,\Psi\left(I\right)\cup\Psi\left(I\right)^{-1}\right)$.
For a similar argument, see \cite{MR2338235}.
\begin{lem}
\label{lem:roughisometric-withcayleygraph}Suppose that $\Psi\left(\Sigma^{*}\right)=G$
and let $j\in\N_{0}$. We then have that the graphs  $X_{j}\left(\Sigma\times G,\sigma\rtimes\Psi\right)$
and $X\negthinspace\left(G,\Psi\left(I\right)\cup\Psi\left(I\right)^{-1}\right)$
are roughly isometric. \end{lem}
\begin{proof}
By identifying $\Sigma^{0}\times G$ with $G$, we clearly have that $X_{0}$ is isometric to $X\!\left(G,\Psi\left(I\right)\cup\Psi\left(I\right)^{-1}\right)$. Suppose now that $j\in\N$.
We show that the map $\pi_{2}:\Sigma^{j}\times G\rightarrow G$,  given
by $\pi_{2}\left(\omega,g\right):=g$, for all $\left(\omega,g\right)\in\Sigma^{j}\times G$,
defines a rough isometry between the metric spaces $\left(\Sigma^{j}\times G,d_{j}\right)$
and $\left(G,d\right)$, where $d_{j}$ denotes the graph metric on
$X_{j}$ and $d$ denotes the graph metric on $X\!\left(G,\Psi\left(I\right)\cup\Psi\left(I\right)^{-1}\right)$.
 Clearly, we have that $\pi_{2}$ is surjective. Further, by the
definition of the edge set of $X_{j}$, we have that if two vertices
$\left(\omega,g\right),\left(\omega',g'\right)\in\Sigma^{j}\times G$
are connected by an edge in $X_{j}$, then $g$ and $g'$ are connected
by an edge in $X\!\left(G,\Psi\left(I\right)\cup\Psi\left(I\right)^{-1}\right)$.
Hence, for all $\left(\omega,g\right),\left(\omega',g'\right)\in\Sigma^{j}\times G$
we have that $d\left(\pi_{2}\left(\omega,g\right),\pi_{2}\left(\omega',g'\right)\right)\le d_{j}\left(\left(\omega,g\right),\left(\omega',g'\right)\right)$.

It remains to show that there exist constants $A,B>0$ such that for
all $\left(\omega,g\right),\left(\omega',g'\right)\in\Sigma^{j}\times G$,
\begin{equation}
d_{j}\left(\left(\omega,g\right),\left(\omega',g'\right)\right)\le Ad\left(\pi_{2}\left(\omega,g\right),\pi_{2}\left(\omega',g'\right)\right)+B.\label{eq:rough-isometry-inequality}
\end{equation}
First note that by our assumptions, there exists a finite set $F\subset\Sigma^{*}$
with the following properties.
\begin{enumerate}
\item For all $\tau\in\Sigma^{j}$ there exists $\kappa\left(\tau\right)\in F$
such that $\Psi\left(\tau\right)\Psi\left(\kappa\left(\tau\right)\right)=\id$,
and for all $h\in\Psi\left(I\right)\cup\Psi\left(I\right)^{-1}$ there
is $\alpha\in F$ such that $\Psi\left(\alpha\right)=h$. (We used
that $\card\left(I\right)<\infty$ and hence, $\card\left(\Sigma^{j}\right)<\infty$,
and that $\Psi\left(\Sigma^{*}\right)=G$.)
\item For all $a,b\in I$ there exists $\gamma\in F\cap\Psi^{-1}\left\{ \id\right\} \cup\left\{ \emptyword\right\} $
such that $a\gamma b\in\Sigma^{*}$. (We used $\card\left(I\right)<\infty$
and item (3) of  our standing assumptions.)
\end{enumerate}
Setting $L:=\max_{\gamma\in F}\left|\gamma\right|$, $A:=2L$ and
$B:=3L+j$, we will show that (\ref{eq:rough-isometry-inequality})
holds. Let $\left(\omega,g\right),\left(\omega',g'\right)\in\Sigma^{j}\times G$
be given. First suppose that $d\left(\pi_{2}\left(\omega,g\right),\pi_{2}\left(\omega',g'\right)\right)=m\in\N$.
Hence, there exist $h_{1},\dots,h_{m}\in\Psi\left(I\right)\cup\Psi\left(I\right)^{-1}$
such that $gh_{1}\cdot\dots\cdot h_{m}=g'$. By property (1) above,
there exist $\alpha_{1},\dots,\alpha_{m}\in F$ such that $\Psi\left(\alpha_{i}\right)=h_{i}$
for all $1\le i\le m$, and there exists $\kappa(\omega)\in F$ such that $\Psi\left(\omega\right)\Psi\left(\kappa\left(\omega\right)\right)=\id$. Then property (2) implies the existence of
$\gamma_{0},\gamma_{1},\dots,\gamma_{m+1}\in F\cap\Psi^{-1}\left\{ \id\right\} \cup\left\{ \emptyword\right\} $
such that $\omega\gamma_{0}\kappa\left(\omega\right)\gamma_{1}\alpha_{1}\gamma_{2}\alpha_{2}\cdot\dots\cdot\gamma_{m}\alpha_{m}\gamma_{m+1}\omega'\in\Sigma^{*}$
and hence,
\[
\left[\omega\gamma_{0}\kappa\left(\omega\right)\gamma_{1}\alpha_{1}\gamma_{2}\alpha_{2}\cdot\dots\cdot\gamma_{m}\alpha_{m}\gamma_{m+1}\omega'\right]\subset\left(\left[\omega\right]\times\left\{ g\right\} \right)\cap\left(\sigma\rtimes\Psi\right)^{-l}\left(\left[\omega'\right]\times\left\{ g'\right\} \right),
\]
where we have set $l:=\left|\omega\gamma_{0}\kappa\left(\omega\right)\gamma_{1}\alpha_{1}\gamma_{2}\alpha_{2}\cdot\dots\cdot\gamma_{m}\alpha_{m}\gamma_{m+1}\right|\le\left(2m+3\right)L+j$.
The inequality in (\ref{eq:rough-isometry-inequality}) follows. Finally,
if $d\left(\pi_{2}\left(\omega,g\right),\pi_{2}\left(\omega',g'\right)\right)=0$
then $g=g'$ and there exist $\gamma_0, \gamma_1 \in F\cap \Psi^{-1}\{\id\}\cup \{\emptyset\}$  such that $\omega \gamma_0 \kappa(\omega) \gamma_1 \omega' \in \Sigma^*$,  which proves  $d_{j}\left(\left(\omega,g\right),\left(\omega',g'\right)\right)\le B$.
 The proof is complete.
\end{proof}
In the following proposition we let $\mathbb{E}\left(\cdot|\mathcal{C}(j)\right):L^{2}\left(\Sigma\times G,\mu_{\varphi}\times\lambda\right)\rightarrow V_{j}$
denote the conditional expectation given $\mathcal{C}(j)$. 
\begin{prop}
\label{pro:transitionmatrix-via-pfadjoint}Suppose that $\Psi\left(\Sigma^{*}\right)=G$.
Let $\varphi:\Sigma\rightarrow\R$ be $\mathcal{C}(k)$-measurable
for some $k\in\N_{0}$, such that $\mathcal{L}_{\varphi}\1=\1$. The
following holds for all $j\in\N$ with $j\ge k-1$. For the bounded
linear operator $\mathbb{E}\left(U\left(\cdot\right)|\mathcal{C}\left(j\right)\right):V_{j}\rightarrow V_{j}$
we have that
\[
\rho\left(\mathbb{E}\left(U\left(\cdot\right)|\mathcal{C}\left(j\right)\right)\right)\le\Vert\mathbb{E}\left(U\left(\cdot\right)|\mathcal{C}\left(j\right)\right)\Vert=1
\]
with equality if and only if $G$ is amenable. In particular, we have that
\[
\rho\left(\mathcal{L}_{\varphi\circ\pi_{1}}\big|_{V_{j}}\right)\le\Vert\mathcal{L}_{\varphi\circ\pi_{1}}\big|_{V_{j}}\Vert=1
\]
with equality if and only if $G$ is amenable. \textup{\emph{}}\end{prop}
\begin{proof}
Fix $j\in\N$ with $j\ge k-1$. We first observe that for each $f\in V_{j}$
we have that $\mathbb{E}\left(U\left(f\right)|\mathcal{C}\left(j\right)\right)$
is the unique element in $V_{j}$,  such that $\left(\mathbb{E}\left(U\left(f\right)|\mathcal{C}\left(j\right)\right),g\right)=\left(U\left(f\right),g\right)$  for all $g\in V_{j}$.
Since  $\left(U\left(f\right),g\right)=\left(f,\mathcal{L}_{\varphi\circ\pi_{1}}\left(g\right)\right)$ and  $V_{j}$ is $\mathcal{L}_{\varphi\circ\pi_{1}}$-invariant
by Lemma \ref{lem:vp_invariantsubspaces},  we conclude that $\mathbb{E}\left(U\left(\cdot\right)|\mathcal{C}\left(j\right)\right)$
is the adjoint of $\mathcal{L}_{\varphi\circ\pi_{1}}\big|_{V_{j}}$.
Since $U\left(V_{0}\right)\subset V_{1}\subset V_{j}$, we have that
the restriction of $\mathbb{E}\left(U\left(\cdot\right)|\mathcal{C}\left(j\right)\right)$
to $V_{0}$ is equal to $U\big|_{V_{0}}$. Because $U$ is an isometry
by Lemma \ref{fac:pf-fact} (\ref{enu:UisIsometry}), we conclude that
$\Vert\mathcal{L}_{\varphi\circ\pi_{1}}\big|_{V_{j}}\Vert=\Vert\mathbb{E}\left(U\left(\cdot\right)|\mathcal{C}\left(j\right)\right)\Vert=1$.

In order to prove the amenability dichotomy for $\rho\left(\mathbb{E}\left(U\left(\cdot\right)|\mathcal{C}\left(j\right)\right)\right)$
we aim to apply Theorem \ref{thm:woess-amenability-randomwalk-characterization}
to a transition matrix on the vertex set $\Sigma^{j}\times G$ of
the graph $X_{j}$. Since $\left\{ \1_{\left[\omega\right]\times\left\{ g\right\} }:\left(\omega,g\right)\in\Sigma^{j}\times G\right\} $
is a basis of $V_{j}$, we obtain a Hilbert space isomorphism between
$V_{j}$ and $\ell^{2}\left(\Sigma^{j}\times G,\nu_{j}\right)$ by
setting $\nu_{j}\left(\omega,g\right):=\left(\mu_{\varphi}\times\lambda\right)\left(\left[\omega\right]\times\left\{ g\right\} \right)$
for every $\left(\omega,g\right)\in\Sigma^{j}\times G$. Using this
isomorphism and with respect to the canonical basis of $\ell^{2}\left(\Sigma^{j}\times G,\nu_{j}\right)$,
we have that $\mathbb{E}\left(U\left(\cdot\right)|\mathcal{C}\left(j\right)\right)$
is represented by the matrix $P=\left(p\left(\left(\omega,g\right),\left(\omega',g'\right)\right)\right)$
given by
\begin{equation}
p\left(\left(\omega,g\right),\left(\omega',g'\right)\right)=\left(U\1_{\left[\omega'\right]\times\left\{ g'\right\} },\1_{\left[\omega\right]\times\left\{ g\right\} }\right)\left(\left(\mu_{\varphi}\times\lambda\right)\left(\left[\omega\right]\times\left\{ g\right\} \right)\right)^{-1}.\label{eq:transitionmatrix-representation}
\end{equation}
Note that we have chosen the matrix $P$ to act on the left. Summing
over $\left(\omega',g'\right)\in\Sigma^{j}\times G$ in the previous
line, we obtain  that $P$ is a transition matrix on $\Sigma^{j}\times G$.
Using that $\mu_{\varphi}\times\lambda$ is $\left(\sigma\rtimes\Psi\right)$-invariant
by Lemma \ref{lem:mu-phi-prod-counting-is-invariant}, one then deduces
from (\ref{eq:transitionmatrix-representation}) that $\nu_{j}$ is
$P$-invariant. Let us now verify that Theorem \ref{thm:woess-amenability-randomwalk-characterization}
is applicable to the transition matrix $P$ acting on the vertex set
$\Sigma^{j}\times G$ of $X_{j}$. Since $\card\left(I\right)<\infty$, 
we have that $X_{j}$ has bounded geometry. Further, it follows immediately
from the definition of $X_{j}$ that $p\left(\left(\omega,g\right),\left(\omega',g'\right)\right)>0$
implies that $\left(\omega,g\right)\sim\left(\omega',g'\right)$ in
$X_{j}$ and hence, $P$ has bounded range ($R=1$) with respect to
$X_{j}$. It is  also clear from the definition of $\nu_{j}$ that
\[
0<\min_{\omega\in\Sigma^{j}}\mu_{\varphi}\left(\left[\omega\right]\right)=\inf_{\left(\omega,g\right)\in\Sigma^{j}\times G}\nu_{j}\left(\omega,g\right)\le\sup_{\left(\omega,g\right)\in\Sigma^{j}\times G}\nu_{j}\left(\omega,g\right)=\max_{\omega\in\Sigma^{j}}\mu_{\varphi}\left(\left[\omega\right]\right)<\infty.
\]
It remains to verify that $P$ is uniformly irreducible with respect
to $X_{j}$. Let $\left(\omega,g\right),\left(\omega',g'\right)\in\Sigma^{j}\times G$
denote a pair of vertices which is connected by an edge in $X_{j}$.
By definition, we then have that $\left(\sigma\rtimes\Psi\right)^{-1}\left(\left[\omega'\right]\times\left\{ g'\right\} \right)\cap\left(\left[\omega\right]\times\left\{ g\right\} \right)\neq\emptyset\mbox{ or }\left(\sigma\rtimes\Psi\right)^{-1}\left(\left[\omega\right]\times\left\{ g\right\} \right)\cap\left(\left[\omega'\right]\times\left\{ g'\right\} \right)\neq\emptyset$.
In the first case, we have that
\begin{eqnarray*}
p\left(\left(\omega,g\right),\left(\omega',g'\right)\right) & = & \left(\mu_{\varphi}\times\lambda\right)\left(\left(\sigma\rtimes\Psi\right)^{-1}\left(\left[\omega'\right]\times\left\{ g'\right\} \right)\cap\left(\left[\omega\right]\times\left\{ g\right\} \right)\right)\left(\mu_{\varphi}\left(\left[\omega\right]\right)\right)^{-1}\\
 & = & \mu_{\varphi}\left(\left[\omega \omega'_j\right]\right)\left(\mu_{\varphi}\left(\left[\omega\right]\right)\right)^{-1}\ge\min_{\tau\in\Sigma^{j+1}}\mu_{\varphi}\left(\left[\tau\right]\right)>0.
\end{eqnarray*}
Next we consider the second case in which $\left(\sigma\rtimes\Psi\right)^{-1}\left(\left[\omega\right]\times\left\{ g\right\} \right)\cap\left(\left[\omega'\right]\times\left\{ g'\right\} \right)\neq\emptyset$
and thus, $g'\Psi\left(\omega'_{1}\right)=g$. Similarly as in the
proof of Lemma \ref{lem:roughisometric-withcayleygraph} one can verify
that there exists a finite set $F\subset\Sigma^{*}$ with the following
properties. Firstly, for all $\tau\in\Sigma^{j}\cup I$ there exists
$\kappa\left(\tau\right)\in F$ such that $\Psi\left(\tau\right)\Psi\left(\kappa\left(\tau\right)\right)=\id$
and secondly, for all $a,b\in I$ there exists $\gamma\in F\cap\Psi^{-1}\left\{ \id\right\} \cup\left\{ \emptyword\right\} $
such that $a\gamma b\in\Sigma^{*}$. Hence, there exist $\gamma_{1},\gamma_{2},\gamma_{3}\in F$
such that
\[
\left(\left[\omega\gamma_{1}\kappa\left(\omega\right)\gamma_{2}\kappa\left(\omega'_{1}\right)\gamma_{3}\omega'\right]\times\left\{ g\right\} \right)\subset\left(\left[\omega\right]\times\left\{ g\right\} \right)\cap\left(\sigma\rtimes\Psi\right)^{-l}\left(\left[\omega'\right]\times\left\{ g'\right\} \right),
\]
where we have set $l:=\left|\omega\gamma_{1}\kappa\left(\omega\right)\gamma_{2}\kappa\left(\omega'_{1}\right)\gamma_{3}\right|\le j+5\max_{\gamma\in F}\left|\gamma\right|$.
Consequently, 
\[
p^{\left(l\right)}\left(\left(\omega,g\right),\left(\omega',g'\right)\right)\ge \left(\min_{\tau\in\Sigma^{j+1}}\mu_{\varphi}\left(\left[\tau\right]\right)\right)^{j+5\max_{\gamma\in F}\left|\gamma\right|} >0.
\]
Hence, with $K:=j+5\max_{\gamma\in F}\left|\gamma\right|$ and $\epsilon:=\left(\min_{\tau\in\Sigma^{j+1}}\mu_{\varphi}\left(\left[\tau\right]\right)\right)^{j+5\max_{\gamma\in F}\left|\gamma\right|}  >0$
we have that $P$ is uniformly irreducible with respect $X_{j}$.

We are now in the position to apply Theorem \ref{thm:woess-amenability-randomwalk-characterization} to the transition matrix $P$, 
which gives  that $\rho\left(P\right)=1$ if and only if $X_{j}$
is amenable. Since $X_j$ is roughly isometric to the Cayley graph of $G$ with respect to $\Psi\left(I\right)\cup\Psi\left(I\right)^{-1}$ by Lemma \ref{lem:roughisometric-withcayleygraph}, it follows from Theorem \ref{thm:amenability-is-roughisometry-invariant} that $X_j$ is amenable if and only if  $G$ is amenable (cf. Proposition \ref{pro:groupamenable-iff-graphamenable}) .  Finally, since $\mathbb{E}\left(U\left(\cdot\right)|\mathcal{C}\left(j\right)\right)$
and $P$ are conjugated by an isomorphism of Hilbert spaces, we have  $\rho\left(\mathbb{E}\left(U\left(\cdot\right)|\mathcal{C}\left(j\right)\right)\right)=\rho(P)$, which  completes the proof.
\end{proof}
Summarizing the outcomes of this section, we obtain the following
main result.
\begin{thm}
\label{thm:amenability-dichotomy-markov}Suppose that $\Psi\left(\Sigma^{*}\right)=G$
and let $\varphi:\Sigma\rightarrow\R$ be $\mathcal{C}(k)$-measurable
for some $k\in\N_{0}$. The following holds for all $j\in\N$ with
$j\ge k-1$. We have
\begin{equation}
\mathcal{P}\left(\varphi,\Psi^{-1}\left\{ \id\right\} \cap\Sigma^{*}\right)\le\log\rho\left(\mathcal{L}_{\varphi\circ\pi_{1}}\big|_{V_{j}}\right)\le\log\rho\left(\mathcal{L}_{\varphi\circ\pi_{1}}\right)=\mathcal{P}\left(\varphi\right),\label{eq:amenability-dichotomy-1}
\end{equation}
with equality in the second inequality if and only if $G$ is amenable.
Moreover, if $\varphi$ is asymptotically  symmetric with respect to $\Psi$,
then
\begin{equation}
\mathcal{P}\left(\varphi,\Psi^{-1}\left\{ \id\right\} \cap\Sigma^{*}\right)=\log\rho\left(\mathcal{L}_{\varphi\circ\pi_{1}}\big|_{V_{j}}\right)\label{eq:amenability-dicotomy-2}
\end{equation}
and so, $G$ is amenable if and only if $\mathcal{P}\left(\varphi,\Psi^{-1}\left\{ \id \right\} \cap\Sigma^{*}\right)=\mathcal{P}\left(\varphi\right)$. \end{thm}
\begin{proof}
Fix $j\in\N$ with $j\ge k-1$,  which implies that $V_{j}$ is $\mathcal{L}_{\varphi\circ\pi_{1}}$-invariant
by Lemma \ref{lem:vp_invariantsubspaces}. As shown in the proof of
Corollary \ref{cor:pressure-is-spectralradius} we may assume without
loss of generality that $\mathcal{L}_{\varphi}\1=\1$ and thus $\mathcal{P}\left(\varphi\right)=0$.

The first inequality in (\ref{eq:amenability-dichotomy-1}) follows
from Corollary \ref{cor:upperboundviaspectralradius} applied to $V=V_{j}$.
The second inequality in (\ref{eq:amenability-dichotomy-1}) is an
immediate consequence of the definition of the spectrum. The amenability
dichotomy follows from Proposition \ref{pro:transitionmatrix-via-pfadjoint}.
The equality $\log\rho\left(\mathcal{L}_{\varphi\circ\pi_{1}}\right)=\mathcal{P}\left(\varphi\right)$
follows from Lemma  \ref{fac:pf-fact} (\ref{enu:.pf-fact-spectralradius-pressure}).

In order to complete the proof,  we now address (\ref{eq:amenability-dicotomy-2})
under the assumption that $\varphi$ is asymptotically  symmetric with respect
to $\Psi$. By Corollary \ref{cor:pressure-is-spectralradius},  we
then have that
\[
\sup_{g\in G}\left\{ \mathcal{P}\left(\varphi,\Psi^{-1}\left\{ g\right\} \cap\Sigma^{*}\right)\right\} =\log\rho\left(\mathcal{L}_{\varphi\circ\pi_{1}}\big|_{V_{j}}\right).
\]
Using that $\Psi\left(\Sigma^{*}\right)=G$ and item (3) of our standing assumptions, one easily verifies that the pressure  $\mathcal{P}\left(\varphi,\Psi^{-1}\left\{ g\right\} \cap\Sigma^{*}\right)$
is independent of $g\in G$, which completes the proof.
\end{proof}

\begin{cor}
\label{cor:amenable-implies-fullpressure}
Let $\varphi:\Sigma\rightarrow\R$ be $\mathcal{C}(k)$-measurable, 
for some $k\in\N_{0}$ and assume that $\varphi$ is asymptotically symmetric with respect to $\Psi$. If $G$ is amenable, then $\mathcal{P}\left(\varphi,\Psi^{-1}\left\{ \id \right\} \cap\Sigma^{*}\right)=\mathcal{P}\left(\varphi\right)$.
\end{cor}
\begin{proof}
Using item (3) of our standing assumptions and that $\varphi$ is asymptotically symmetric with respect to $\Psi$, one verifies that $G':=\Psi(\Sigma^*)$ is a subgroup of $G$. Since $G$ is amenable, it is well-known that also $G'$ is amenable (see e.g. \cite[Theorem 12.2 (c)]{MR1743100}), and the corollary follows from Theorem \ref{thm:amenability-dichotomy-markov}.
\end{proof}

\begin{rem}
\label{proof-comment-stadlabuer}It is not difficult to extend Corollary  \ref{cor:amenable-implies-fullpressure} to arbitrary  H\"older
continuous potentials by approximating a  H\"older
continuous potential by a $\mathcal{C}\left(k\right)$-measurable
potential and then  letting $k$ tend to infinity.  One obtains that,  for an amenable group 
$G$ and for an asymptotically symmetric   H\"older continuous potential $\varphi$,  we have  $\mathcal{P}\left(\varphi,\Psi^{-1}\left\{ \id\right\} \cap\Sigma^{*}\right)=\mathcal{P}\left(\varphi\right)$.  This was proved by the author in  \cite[Theorem 5.3.11]{JaerischDissertation11}, and independently, by Stadlbauer  \cite[Theorem 4.1]{Stadlbauer11} in a slightly different setting.
The reversed implication of Corollary  \ref{cor:amenable-implies-fullpressure} 
was proved recently in \cite[Theorem 5.4]{Stadlbauer11} by extending ideas of Day (\cite{MR0159230}).  A generalization of (\ref{eq:amenability-dicotomy-2}) in Theorem \ref{thm:amenability-dichotomy-markov} for arbitrary H\"older continuous potentials seems still to be open. 
\end{rem}

\section{Proof of the Main Results\label{sec:Proofs}}

For a linear GDMS $\Phi$ associated to $\F_{d}=\langle g_{1},\dots,g_{d}\rangle$,
$d\ge2$, we set $I:=\left\{ g_1,g_{1}^{-1},\dots,g_d,g_{d}^{-1}\right\} $
and we consider the Markov shift $\Sigma$,  given by
\[
\Sigma:=\left\{ \omega\in I^{\N}:\,\,\omega_{i}\neq\left(\omega_{i+1}\right)^{-1}\,\,\mbox{for all }i\in\N\right\} .
\]
The involution $\kappa:\Sigma^{*}\rightarrow\Sigma^{*}$ is given
by $\kappa\left(\omega\right):=\left(\omega_{n}^{-1},\omega_{n-1}^{-1},\dots,\omega_{1}^{-1}\right)$,
for all $n\in\N$ and $\omega\in\Sigma^{n}$.

For a normal subgroup $N$ of $\F_{d}$,  we let $\Psi_{N}:I^{*}\rightarrow\F_{d}/N$
denote the unique semigroup homomorphism such that $\Psi_{N}\left(g\right)=g\mbox{ mod }N$
for all $g\in I$. Clearly, we have that
\begin{equation}
\Psi_{N}\left(\Sigma^{*}\right)=\F_{d}/N.\label{eq:proof-psi-onto}
\end{equation}
Since the assertions in Theorem \ref{thm:lineargdms-amenability-dichotomy} and Proposition \ref{pro:lineargdms-brooks} are clearly satisfied in the case that $N=\left\{ \id\right\}$, we will from now on  assume that $N\neq\left\{ \id\right\} $. Using that $N$ is a  normal subgroup of $\F_{d}$ and $d\ge2$, one easily verifies that there exists a
finite set $F\subset\Sigma^{*}\cap\Psi_{N}^{-1}\left\{ \id\right\} $
with the following property:
\begin{equation}
\mbox{For all }i,j\in I\mbox{ there exists }\tau\in F\cup\left\{ \emptyword\right\} \mbox{ such that }i\tau j\in\Sigma^{*}.\label{eq:proof-mixingproperty}
\end{equation}
Note that (\ref{eq:proof-mixingproperty}) implies that the group-extended Markov system $\left(\Sigma\times\left(\F_{d}/N\right),\sigma\rtimes\Psi_{N}\right)$
 satisfies  item (3) of our standing
assumptions at the beginning of Section 4. Hence, the results of
Section 3 are applicable to  the $\mathcal{C}\left(1\right)$-measurable potential $\varphi:\Sigma\rightarrow\R$,  given by  $\varphi_{|\left[g\right]}=\log\left(c_{\Phi}\left(g\right)\right)$
for all $g\in I$.
\begin{proof}
[Proof of Theorem \ref{thm:lineargdms-amenability-dichotomy} ]Our
aim is to apply Theorem \ref{thm:amenability-dichotomy-markov} to
the group-extended Markov system $\left(\Sigma\times\left(\F_{d}/N\right),\sigma\rtimes\Psi_{N}\right)$
and the $\mathcal{C}\left(1\right)$-measurable potential $s\varphi:\Sigma\rightarrow\R$,
for each $s\in\R.$ By (\ref{eq:proof-psi-onto}) and (\ref{eq:proof-mixingproperty}), 
we are left to show that $s\varphi$ is asymptotically  symmetric with respect
to $\Psi_{N}$. Since $\Phi$ is symmetric we
have that $c_{\Phi}\left(\omega\right)=c_{\Phi}\left(\kappa\left(\omega\right)\right)$,
for all $\omega\in\Sigma^{*}$. Hence, for all $s\in\R$, $n\in\N$
and $g\in \F_d / N$,  we have that
\begin{eqnarray*}
\sum_{\omega\in\Sigma^{n}:\,\Psi_{N}\left(\omega\right)=g}\exp\left(sS_{\omega}\varphi\right) & = & \sum_{\omega\in\Sigma^{n}:\,\Psi_{N}\left(\omega\right)=g}\left(c_{\Phi}\left(\omega\right)\right)^{s}=\sum_{\omega\in\Sigma^{n}:\,\Psi_{N}\left(\omega\right)=g}\left(c_{\Phi}\left(\kappa\left(\omega\right)\right)\right)^{s}\\
 & = & \sum_{\omega\in\Sigma^{n}:\,\Psi_{N}\left(\omega\right)=g^{-1}}\left(c_{\Phi}\left(\omega\right)\right)^{s}=\sum_{\omega\in\Sigma^{n}:\,\Psi_{N}\left(\omega\right)=g^{-1}}\exp\left(sS_{\omega}\varphi\right),
\end{eqnarray*}
which proves that $s\varphi$ is asymptotically  symmetric with respect to
$\Psi_{N}$. We are now in the position to apply Theorem \ref{thm:amenability-dichotomy-markov},
which gives that amenability of $\F_{d}/N$ is
equivalent to
\[
\mathcal{P}\left(s\varphi,\Psi_{N}^{-1}\left\{ \id\right\} \cap\Sigma^{*}\right)=\mathcal{P}\left(s\varphi\right).
\]
Since 
$\delta\left(N,\Phi\right)$ is equal to the unique zero of $s\mapsto\mathcal{P}\left(s\varphi,\Psi_{N}^{-1}\left\{ \id\right\} \cap\Sigma^{*}\right)$
and  $\delta\left(\F_{d},\Phi\right)$ is  equal to the unique zero of $s\mapsto\mathcal{P}\left(s\varphi\right)$ by Fact \ref{fac:criticalexponents-via-pressure}, we conclude that
\[
\delta\left(\F_{d},\Phi\right)=\delta\left(N,\Phi\right)\,\,\,\textrm{if and only if }\F_{d}/N\textrm{ is amenable}.
\]
The proof is complete.
\end{proof}
For the proof of Theorem \ref{thm:lineargdms-lowerhalfbound} we need
the following lemma.
\begin{lem}
\label{lem:delta-half-divergencetype}Let $\Phi$ be a  symmetric linear GDMS
 associated to $\F_{d}$. For every  non-trivial normal subgroup
$N$ of $\F_{d}$,  we have that 
\[
\sum_{h\in N}\left(c_{\Phi}\left(h\right)\right)^{\delta\left(\F_{d},\Phi\right)/2}=\infty.
\]
In particular, we have that  $\delta\left(N,\Phi\right)\ge\delta\left(\F_{d},\Phi\right)/2$. \end{lem}
\begin{proof}
First observe that $N$ and $\Psi_{N}^{-1}\left\{ \id\right\} \cap\Sigma^{*}$
are in one-to-one correspondence, which implies that  
\[
\sum_{h\in N}\left(c_{\Phi}\left(h\right)\right)^{\delta\left(\F_{d},\Phi\right)/2}=\sum_{\omega\in\Psi_{N}^{-1}\left\{ \id\right\} \cap\Sigma^{*}}\exp\left(\left(\delta\left(\F_{d},\Phi\right)/2\right)S_{\omega}\varphi\right).
\]
For each $\omega\in\Sigma^{*}$,  we can choose $\tau\left(\omega\right)\in F$
such that $\omega\tau\left(\omega\right)\kappa\left(\omega\right)\in\Sigma^{*}$
by making use of property (\ref{eq:proof-mixingproperty}). Further, we  define the
map $\Theta:\Sigma^{*}\rightarrow\Psi_{N}^{-1}\left\{ \id\right\} \cap\Sigma^{*}$, 
 $\Theta(\omega):=\omega\tau\left(\omega\right)\kappa\left(\omega\right)$,
which is at most $\card\left(F\right)$-to-one. Moreover, setting  $C:=\min\left\{ S_{\tau}\varphi/2:\tau\in F\right\} >-\infty$
and using that $\Phi$ is symmetric,   we observe that $S_{\omega}\varphi+C=S_{\omega}\varphi/2+S_{\kappa\left(\omega\right)}\varphi/2+C\le S_{\Theta\left(\omega\right)}\varphi/2$,  for each $\omega\in\Sigma^{*}$.
Consequently, we have that
\begin{eqnarray}
\sum_{\omega\in\Psi_{N}^{-1}\left\{ \id\right\} \cap\Sigma^{*}}\exp\left(\left(\delta\left(\F_{d},\Phi\right)/2\right)S_{\omega}\varphi\right) & \ge &\card\left(F\right)^{-1}\sum_{\omega\in\Sigma^{*}}\exp\left(\left(\delta\left(\F_{d},\Phi\right)/2\right)S_{\Theta\left(\omega\right)}\varphi\right)\label{eq:delta-half-bound-1}\\
 & \ge & \card\left(F\right)^{-1} \exp\left({\delta\left(\F_{d},\Phi\right)C}\right)\sum_{\omega\in\Sigma^{*}}\exp\left(\delta\left(\F_{d},\Phi\right)S_{\omega}\varphi\right).\nonumber
\end{eqnarray}
Finally, the existence of the Gibbs measure $\mu=\mu_{\delta\left(\F_{d},\Phi\right)\varphi}$ implies that there exists  a constant $C_{\mu}>0$  such that
\[
\sum_{\omega\in\Sigma^{*}}\exp\left(\delta\left(\F_{d},\Phi\right)S_{\omega}\varphi\right)\ge C_{\mu}\sum_{\omega\in\Sigma^{*}}\mu\left(\left[\omega\right]\right)=C_{\mu}\sum_{n\in\N}\sum_{\omega\in\Sigma^{n}}\mu\left(\left[\omega\right]\right)=C_{\mu} \sum_{n\in\N}1=\infty.
\]
Combining the latter  estimate  with (\ref{eq:delta-half-bound-1}),  the proof is complete.
\end{proof}
\begin{proof}
[Proof of Theorem \ref{thm:lineargdms-lowerhalfbound}] By Theorem
\ref{thm:lineargdms-amenability-dichotomy},  the assertion is clearly
true if $\F_{d}/N$ is amenable. We address the remaining  case that $\F_{d}/N$
is non-amenable. Suppose  for a contradiction that the claim  is wrong.
By Lemma \ref{lem:delta-half-divergencetype},  we obtain that 
\begin{equation}
\delta\left(N,\Phi\right)=\delta\left(\F_{d}\right)/2.\label{eq:deltaN-is-deltahalf}
\end{equation}
For notational convenience,  we set $G:=\F_{d}/N$ throughout this proof.

Consider the non-negative matrix $P\in\R^{\left(I\times G\right)\times\left(I\times G\right)}$, 
given by
\[
p\left(\left(v_{1},g_{1}\right),\left(v_{2},g_{2}\right)\right)=\begin{cases}
c_{\Phi}\left(v_{1}\right)^{\delta\left(N,\Phi\right)}, & \mbox{if }v_{1}\neq v_{2}^{-1}\mbox{ and }g_{2}=g_{1}\Psi_{N}\left(v_{1}\right)\\
0 & \mbox{else.}
\end{cases}.
\]
By the assertions in  (\ref{eq:proof-psi-onto}) and (\ref{eq:proof-mixingproperty}), 
we have that $P$ is irreducible in the sense that,  for all $x,y\in I\times G$
there exists  $n\in\N$ such that $p^{\left(n\right)}\left(x,y\right)>0$.
Using the irreducibility of $P$ and that $\card\left(I\right)=2d<\infty$, 
we deduce from (\ref{eq:deltaN-is-deltahalf}) and Lemma \ref{lem:delta-half-divergencetype}
that $P$ is $R$-recurrent with $R=1$ in the sense of Vere-Jones
(\cite{MR0141160}, see also Seneta \cite[Definition 6.4]{MR2209438}). That is,  $P$ satisfies the following properties. 
\begin{equation}
\limsup_{n\rightarrow\infty}\left(p^{\left(n\right)}\left(x,y\right)\right)^{1/n}=1\mbox{ and }\sum_{n\in\N}p^{\left(n\right)}\left(x,y\right)=\infty,\mbox{ for all }x,y\in I\times G.\label{eq:recurrent-matrix}
\end{equation}
Thus, by \cite[Theorem 6.2]{MR2209438}, it follows that there exists
a positive row vector $h\in\R^{I\times G}$ such that
\begin{equation}
hP=h.\label{eq:strict-half-bound-1a}
\end{equation}
It also follows from \cite[Theorem 6.2]{MR2209438} that the vector  $h$ in (\ref{eq:strict-half-bound-1a}) is unique up to a constant multiple.
Next, we define the non-negative matrix $P_{h}\in\R^{\left(I\times G\right)\times\left(I\times G\right)}$, which is for all $x,y \in I\times G$ given by 
\[
p_{h}\left(x,y\right)=p\left(y,x\right)h\left(y\right)/h\left(x\right).
\]
It follows from (\ref{eq:strict-half-bound-1a}) that $P_{h}$ is
a transition matrix on $I\times G$. Further,  we deduce from   (\ref{eq:recurrent-matrix}) that $P_{h}$ is $1$-recurrent.

In order to derive a contradiction,  we consider $P_{h}$ as a random
walk on the graph $X_{1}$ associated to the group-extended Markov
system $\left(\Sigma\times G,\sigma\rtimes\Psi_{N}\right)$ (see Definition
\ref{def:k-cylinder-graphs}), and we investigate the automorphisms
of $X_{1}$. Let $\Aut\left(X_{1}\right)$ denote the group of self-isometries
of $\left(X_{1},d_{X_{1}}\right)$,  where $d_{X_{1}}$ denotes the
graph metric on $X_{1}$. Note that each element $g\in G$ gives rise
to an automorphism $\gamma_{g}\in\Aut\left(X_{1}\right)$,  which is
 given by $\gamma_{g}\left(i,\tau\right):=\left(i,g\tau\right)$, for each $\left(i,\tau\right)\in I\times G$.
The next step is to verify that also $\gamma_{g}\in\Aut\left(X_{1},P_{h}\right)$, 
where we have set
\[
\Aut\left(X_{1},P_{h}\right):=\left\{ \gamma\in\Aut\left(X_{1}\right):P_{h}\left(x,y\right)=P_{h}\left(\gamma x,\gamma y\right),\mbox{ for all }x,y\in I\times G\right\} .
\]
Since $P$ has the property that $p\left(x,y\right)=p\left(\gamma_{g}\left(x\right),\gamma_{g}\left(y\right)\right)$, 
for all $x,y\in I\times G$ and $g\in G$, it follows that the vector  $h_{g}\in\R^{I\times G}$, 
given by $h_{g}\left(i,\tau\right):=h\left(i,g\tau \right)$, $\left(i,\tau\right)\in I\times G$,
satisfies $h_{g}P=h_{g}$ as well. Since the function $h$ in (\ref{eq:strict-half-bound-1a})
is unique up to a constant multiple, we conclude that there exists
a homomorphism $r:G\rightarrow\R^{+}$ such that$h_{g}=r\left(g\right)h$, 
  for each $g\in G$. Consequently, we have $p_{h}\left(x,y\right)=p_{h}\left(\gamma_{g}\left(x\right),\gamma_{g}\left(y\right)\right)$
for all $x,y\in I\times G$ and $g\in G$.
Hence,  $\gamma_{g}\in\Aut\left(X_{1},P_{h}\right)$
for each $g\in G$. Since $\card(I)<\infty$, we deduce  that $\Aut\left(X_{1},P_{h}\right)$) acts with finitely many orbits on $X_{1}$.

In the terminology of \cite{MR1743100} this is to say that $\left(X_{1},P_{h}\right)$
is a quasi-transitive recurrent random walk. By \cite[Theorem 5.13]{MR1743100}
we then have that $X_{1}$ is a generalized lattice of dimension one
or two. In particular, we have
that $X_{1}$ has polynomial growth with degree one or two (\cite[Proposition 3.9]{MR1743100}). Since
$X_{1}$ is roughly isometric to the Cayley graph of $G$ by Lemma
\ref{lem:roughisometric-withcayleygraph}, we conclude that also  $G$ has
polynomial growth (see e.g. \cite[Lemma 3.13]{MR1743100}).
This contradicts the well-known fact that each non-amenable group
has exponential growth. The proof is complete. \end{proof}
\begin{rem*}
The construction of the matrix $P_{h}$ and the verification of its
invariance properties is analogous to the discussion of the $h$-process
in \cite[Proof of Theorem 7.8]{MR1743100} and goes back to the work
of Guivarc'h (\cite[page 85]{MR588157}) on random walks on groups.
However, note that in our case $P$ is in general not stochastic.
\end{rem*}
\begin{proof}
[Proof of Proposition \ref{pro:lineargdms-brooks}]In order to investigate the radial limit sets of $N$, we introduce an induced  GDMS $\tilde{\Phi}$, whose edge set consists of first return loops in the Cayley graph of $\F_d / N$.   We define $\tilde{\Phi}:=\left(V,\left(X_{v}\right)_{v\in V},\tilde{E},\tilde{i},\tilde{t},\left(\tilde{\phi}_{\omega}\right)_{\omega\in\tilde{E}},\tilde{A}\right)$
as follows. The edge set $\tilde{E}$ and   $\tilde{i},\tilde{t}:\tilde{E}\rightarrow V$
are given by
\[
\tilde{E}:=\left\{ \omega=\left(v_{i},w_{i}\right)\in\Sigma_{\Phi}^{*}:\,\, v_{1}\cdot\dots\cdot v_{\left|\omega\right|}\in N,\,\, v_{1}\cdot\dots\cdot v_{k}\notin N\mbox{ for all }1\le k<\left|\omega\right|\right\} ,
\]
\[
\tilde{i}\left(\omega\right):=i\left(\omega_{1}\right),\,\,\tilde{t}\left(\omega\right):=t\left(\omega_{\left|\omega\right|}\right),\,\,\omega\in\tilde{E},
\]
the matrix $\tilde{A}=\left(\tilde{a}\left(\omega,\omega'\right)\right)\in\left\{ 0,1\right\} ^{\tilde{E}\times\tilde{E}}$
satisfies $\tilde{a}\left(\omega,\omega'\right)=1$ if and only if
$a\left(\omega_{\left|\omega\right|},\omega'_{1}\right)=1$, and the
family $\left(\tilde{\phi}_{\omega}\right)_{\omega\in\tilde{E}}$
is given by $\tilde{\phi}_{\omega}:=\phi_{\omega},\,\,\omega\in\tilde{E}$.
One immediately verifies that $\tilde{\Phi}$ is a conformal GDMS. Note that there are canonical embeddings from $\Sigma_{\tilde{\Phi}}$  into $\Sigma_{\Phi}$ and from $\Sigma_{\tilde{\Phi}}^{*}$ into $\Sigma_{\Phi}^{*}$, which we will both  indicate by omitting the tilde, that is $\tilde{\omega}\mapsto\omega$. For the coding maps $\pi_{\tilde{\Phi}}:\Sigma_{\tilde{\Phi}}\rightarrow J\left(\tilde{\Phi}\right)$
and $\pi_{\Phi}:\Sigma_{\Phi}\rightarrow J\left(\Phi\right)$ we have
$\pi_{\tilde{\Phi}}\left(\tilde{\omega}\right)=\pi_{\Phi}\left(\omega\right)$, for each $\tilde{\omega}\in \Sigma_{\tilde{\Phi}}$.
The following relations between the limit set of $\tilde{\Phi}$ and
the radial limit sets of $N$ are straightforward to prove. We have
that
\[
J^{*}\left(\tilde{\Phi}\right)\subset \Lur(N,\Phi)\subset \Lr(N,\Phi) \subset J\left(\tilde{\Phi}\right)\cup\bigcup_{\eta\in\Sigma_{\Phi}^{*},\tilde{\omega}\in\Sigma_{\tilde{\Phi}}:\eta\omega\in\Sigma_{\Phi}}\phi_{\eta}\left(\pi_{\tilde{\Phi}}\left(\tilde{\omega}\right)\right).
\]
Note that the right-hand side in the latter chain of inclusions can be written as a  countable
union of images of $J\left(\tilde{\Phi}\right)$ under Lipschitz continuous
maps. Since Lipschitz continuous maps do not increase Hausdorff dimension
and since Hausdorff dimension is stable under countable unions, we
obtain
\begin{equation}
\dim_{H}\left(J^{*}\left(\tilde{\Phi}\right)\right)\le\dim_{H}\left(\Lur(N,\Phi)\right)\le\dim_{H}\left(\Lr(N,\Phi)\right)\le\dim_{H}\left(J\left(\tilde{\Phi}\right)\right).\label{eq:hausdorffdimension-inequalities}
\end{equation}
Since the incidence matrix of $\tilde{\Phi}$ is finitely irreducible
by property (\ref{eq:proof-mixingproperty}), the generalised Bowen's
formula (Theorem \ref{thm:cgdms-bowen-formula}) implies that $\dim_{H}\left(J^{*}\left(\tilde{\Phi}\right)\right)=\dim_{H}\left(J\left(\tilde{\Phi}\right)\right)$,
so equality holds in (\ref{eq:hausdorffdimension-inequalities}). 

The final step is to show that $\dim_{H}\left(J\left(\tilde{\Phi}\right)\right)=\delta\left(N,\Phi\right)$.
By Theorem \ref{thm:cgdms-bowen-formula} and Fact \ref{fac:criticalexponents-via-pressure}, 
we have
\[
\dim_{H}\left(J\left(\tilde{\Phi}\right)\right)=\mathcal{P}_{-\tilde{\zeta_{\Phi}}}\left(0,\Sigma_{\tilde{\Phi}}^{*}\right)=\inf\left\{ s\in\R:\sum_{\tilde{\omega}\in\Sigma_{\tilde{\Phi}}^{*}}\e^{sS_{\tilde{\omega}}\zeta_{\tilde{\Phi}}}<\infty\right\} .
\]
Since the elements $\tilde{\omega}\in\Sigma_{\tilde{\Phi}}^{*}$ are
in one-to-one correspondence with $\omega\in\mathcal{C}_{N}$, where $\mathcal{C}_{N}$ is  given by
\[
\mathcal{C}_{N}:=\left\{ \omega=\left(v_{i},w_{i}\right)\in\Sigma_{\Phi}^{*}:\,\, v_{1}\cdot\dots\cdot v_{\left|\omega\right|}\in N\right\} ,
\]
and using that $S_{\tilde{\omega}}\zeta_{\tilde{\Phi}}=S_{\omega}\zeta_{\Phi}$
for all $\tilde{\omega}\in\Sigma_{\tilde{\Phi}}^{*}$, we conclude
that
\[
\dim_{H}\left(J\left(\tilde{\Phi}\right)\right)=\inf\left\{ s\in\R:\sum_{\omega\in\mathcal{C}_{N}}\e^{sS_{\omega}\zeta_{\Phi}}<\infty\right\} .
\]
Finally, since the map from $\mathcal{C}_{N}$ onto $N$,  given by
$\omega=\left(\left(v_{1},w_{1}\right),\left(v_{2},w_{2}\right),\dots,\left(v_{n},w_{n}\right)\right)\mapsto v_{1}v_{2}\cdots v_{n}$, for $n\in\N$,
is $\left(2d-1\right)$-to-one, and since  $S_{\omega}\zeta_{\Phi}=c_{\Phi}\left(v_{1}\dots v_{n}\right)$, for all $\omega\in\mathcal{C}_{N}$, 
it follows that
\[
\dim_{H}\left(J\left(\tilde{\Phi}\right)\right)=\inf\left\{ s\in\R:\sum_{g\in N}\left(c_{\Phi}\left(g\right)\right)^{s}<\infty\right\} =\delta\left(N,\Phi\right),
\]
which completes the proof. 
\end{proof}

\section{Kleinian Groups\label{sec:Kleinian-groups}}
In this section we give a more detailed discussion of Kleinian groups and how these relate to the concept of
a GDMS. In particular, in Proposition \ref{pro:canonicalgdms-gives-radiallimitset} we will  give the motivation for our definition of the  radial limit set in the context of a GDMS associated to the free group (see Definition \ref{def:gdms-associated-to-freegroup-and-radiallimitsets}). 

In the following we let $G\subset\mathrm{Con}\left(m\right)$ denote
a non-elementary, torsion-free Kleinian group acting properly discontinuously
on the $\left(m+1\right)$-dimensional hyperbolic space $\mathbb{D}^{m+1}$, where $\mathrm{Con}\left(m\right)$
denotes the set of orientation preserving conformal automorphisms
of $\mathbb{D}^{m+1}$.  The \emph{limit set} $L\left(G\right)$ of $G$
is the set of accumulation points with respect to the Euclidean  topology on  $\R^{m+1}$ of the $G$-orbit of some arbitrary point in $\mathbb{D}^{m+1}$,
that is, for each $z\in\mathbb{D}^{m+1}$ we have that
\[
L\left(G\right)=\overline{G\left(z\right)}\setminus G\left(z\right),
\]
where the closure is taken with respect to the Euclidean topology on $\R^{m+1}$. Clearly, $L\left(G\right)$ is a subset of $\mathbb{S}$.
For more details on Kleinian groups and their limit sets, we refer
to \cite{Beardon,MR959135,MR1041575,MR1638795,MR2191250}.

Let us recall the definition of the following important subsets of $L(G)$, namely the  \emph{radial} and the \emph{uniformly
radial limit set} of $G$. In here, $s_{\xi}\subset\mathbb{D}^{m+1}$
denotes the hyperbolic ray from $0$ to $\xi$ and $B\left(x,r\right):=\left\{ z\in\mathbb{D}^{m+1}:\,\, d\left(z,x\right)<r\right\} \subset\mathbb{D}^{m+1}$
denotes the open hyperbolic ball of radius $r$ centred at $x$, where $d$ denotes the hyperbolic metric on $\mathbb{D}^{m+1}$.
\begin{defn}
\label{def:radiallimitsets-fuchsian}For a Kleinian
group $G$ the \emph{radial} and the \emph{uniformly radial limit
set} of $G$ are given by
\begin{align*}
L_{\mathrm{r}}\left(G\right) & :=\left\{ \xi\in L\left(G\right):\exists c>0\mbox{ such that }s_{\xi}\cap B\left(g\left(0\right),c\right)\neq\emptyset\mbox{ for infinitely many }g\in G\right\} ,\\
\text{and}
\\
L_{\mathrm{ur}}\left(G\right) & :=\left\{ \xi\in L\left(G\right):\exists c>0\mbox{ such that }s_{\xi}\subset\bigcup_{g\in G}B\left(g\left(0\right),c\right)\right\} .
\end{align*}

\end{defn}
A  Kleinian group $G$ is said to be \emph{geometrically finite}
if the action of $G$ on $\mathbb{D}^{m+1}$ admits a fundamental polyhedron
with finitely many sides. We denote by $E_{G}$ the set of points
in $\mathbb{D}^{m+1}$, which lie on a geodesic connecting any two limit
points in $L\left(G\right)$. The \emph{convex hull} of $E_{G}$, which we will denote
by $C_{G}$, is the minimal hyperbolic convex subset of $\mathbb{D}^{m+1}$
containing $E_{G}$. $G$ is called  \emph{convex cocompact} (\cite[page 7]{MR1041575})
if the action of $G$ on $C_{G}$ has a compact fundamental domain
in $\mathbb{D}^{m+1}.$

The following class of Kleinian groups gives the main motivation for
our definition of a GDMS associated to the free group (see also \cite[X.H]{MR959135}).
\begin{defn}
\label{def:kleinian-of-schottkytype}
Let $d\ge2$ and  let $\mathcal{D}:=\{ (D_{n}^j):n\in \{1,\dots,d\}, j\in \{-1,1\}\}$ be a family  of
 pairwise disjoint compact Euclidean balls $D_{n}^j\subset\R^{m+1}$, which  intersect $\mathbb{S}^m$ orthogonally such that  $\diam\left(D_{n}\right)=\diam\left(D_{n}^{-1}\right)$. For each $n\in \{1,\dots d \}$,  let  $g_{n}\in\mathrm{Con}\left(m\right)$
be the unique hyperbolic element such that $g_{n}\left(\mathbb{D}^{m+1}\cap\partial D_{n}^{-1}\right)=\mathbb{D}^{m+1}\cap\partial D_{n}$,
where $\partial D_{n}^j$ denotes the boundary of $D_{n}^j$ with respect
to the Euclidean metric on $\R^{m+1}$. Then $G:=\left\langle g_{1},\dots,g_{d}\right\rangle $
is referred to as  the \emph{Kleinian group of Schottky type} generated
by $\mathcal{D}$.
\end{defn}

Note that  a Kleinian group of Schottky type $G=\left\langle g_{1},\dots,g_{d}\right\rangle $ is algebraically a free group. The following construction of a  particular GDMS associated to the free group $\langle g_{1},\dots,g_{d}\rangle$ is canonical.
\begin{defn}
\label{def:canonical-model-kleinianschottky}Let $G=\langle g_{1},\dots,g_{d}\rangle$ be
a Kleinian group of Schottky type
generated by $\mathcal{D}$. The \emph{canonical GDMS $\Phi_{G}$ associated
 to $G$}  is the GDMS associated to the free group $\langle g_{1},\dots,g_{d}\rangle$
which satisfies   $X_{g_{n}^j}:=\left(\mathbb{D}^{m+1}\cup\mathbb{S}^m\right)\cap D_{n}^j$,
for each  $n\in \{1,\dots,d\}$ and  $j\in \{-1, 1\}$,    and for which the contractions   $\phi_{\left(v,w\right)}:X_{w}\rightarrow X_{v}$ are  given by $\phi_{\left(v,w\right)}:=v_{\big|X_{w}}$,  for each $(v,w)\in E$. \end{defn}
For the following fact we refer to \cite[Theorem 5.1.6]{MR2003772}.
\begin{fact} \label{fac:coding-kleinian-of-schottkytype}
For a Kleinian group of Schottky type $G$ we have that $L\left(G\right)=J\left(\Phi_{G}\right)$.
\end{fact}

\begin{rem} We remark that without our assumption on $G$ that $\diam\left(D_{n}\right)=\diam\left(D_{n}^{-1}\right)$,  for each $n\in \{1,\dots ,d\}$  in Definition \ref{def:kleinian-of-schottkytype}, the generators of the associated GDMS $\Phi_G$ may fail to be contractions. However, in that case, by taking   sufficiently high iterates of the generators, we can pass to a finite index subgroup of $G$, for which there exists  a set  $\mathcal{D}$ as in Definition \ref{def:kleinian-of-schottkytype}. 
\end{rem}

The following brief discussion of the geometry of a Kleinian group
of Schottky type $G$ contains nothing that is not well known, however, the reader might like
to recall a few of its details. Let $\Phi_{G}$ denote the canonical GDMS associated to $G$. Recall that for  the half-spaces $$H_{v}:=\left\{ z\in\mathbb{D}^{m+1}:\,\, d\left(z,0\right)<d\left(z,v\left(0\right)\right)\right\}, \text{ for each }v\in V, $$
we have that the set
\[
F:=\bigcap_{v\in V}H_{v}
\]
is referred to as a \emph{Dirichlet fundamental domain for $G$}. That $F$ is a fundamental domain for $G$ means that $F$ is an open set which satisfies the conditions
\[
\bigcup_{g\in G}g\left(\overline{F}\cap\mathbb{D}^{m+1}\right)=\mathbb{D}^{m+1}\mbox{ and }g\left(F\right)\cap h\left(F\right)=\emptyset\mbox{ for all }g,h\in G\mbox{ with }g\neq h.
\]
For $\omega=\left(v_{k},w_{k}\right)_{k\in\N}\in\Sigma_{\Phi_G}$ and
$\pi_{\Phi_G}\left(\omega\right)=\xi$, we have that the ray $s_{\xi}$
 successively passes through the fundamental domains
$F,v_{1}\left(F\right),v_{1}v_{2}\left(F\right),\dots$.

We also make use of the fact that a Kleinian group of Schottky type $G$ is convex cocompact. This follows from a theorem due to Beardon and Maskit (\cite{MR0333164}, \cite[Theorem 2]{MR2191250}),
since $G$ is geometrically finite and $L\left(G\right)$ contains  no parabolic
fixed points (cf. \cite[Theorem 12.27]{MR2249478}).  Clearly, if  $G$
is convex cocompact, then there exists $R_{G}>0$ such
that
\begin{equation}
C_{G}\cap g\overline{F}\subset B\left(g\left(0\right),R_{G}\right),\mbox{ for all }g\in G.\label{eq:schottky-ex-2}
\end{equation}
In particular, we have that   $L_{\mathrm{ur}}\left(G\right)=L_{\mathrm{r}}\left(G\right)=L\left(G\right)$.

Using the fact that $G$ acts properly discontinuously on $\mathbb{D}^{m+1}$ and that $G$ is convex cocompact, one
easily verifies that for each  $r>0$ there exists
a finite set $\Gamma\subset G$ such that
\begin{equation}
B\left(0,r\right)\cap C_{G}\subset\bigcup_{\gamma\in \Gamma}\gamma \overline{F}.\label{eq:Lur-coding-2}
\end{equation}

The next proposition provides the main motivation for our definition
of the (uniformly) radial limit set of a normal subgroup $N$ of $\F_{d}$
with respect to a GDMS associated to $\F_{d}$.
\begin{prop}
\label{pro:canonicalgdms-gives-radiallimitset}Let $G$ be a Kleinian
group of Schottky type and let $\Phi_{G}$ denote the canonical GDMS
associated to $G$. For every non-trivial normal subgroup $N$ of
$G$,  we  have that
\[
L_{\mathrm{r}}\left(N\right)=\Lr\left(N,\Phi_{G}\right)\mbox{ and }L_{\mathrm{ur}}\left(N\right)=\Lur\left(N,\Phi_{G}\right).
\]
\end{prop}
\begin{proof}
Let us begin by proving that  $\Lur\left(N,\Phi_G\right)\subset L_{\mathrm{ur}}\left(N\right)$. To start, let $\xi\in \Lur\left(N,\Phi_G\right)$ be given. By the definition of $\Lur\left(N,\Phi_G\right)$,
there exists $\omega=\left(v_{k},w_{k}\right)_{k\in\N}\in\Sigma_{\Phi_G}$ and a finite set $\Gamma\subset G$
such that $\pi_{\Phi_G}\left(\omega\right)=\xi$
and  $v_{1}v_{2}\cdots v_{k}\in N\Gamma$, for all $k\in\N$.
Hence, using  (\ref{eq:schottky-ex-2}), it follows that
\[
s_{\xi}\subset\bigcup_{h\in N}\bigcup_{\gamma\in\Gamma}B\left(h\gamma\left(0\right),R_{G}\right).
\]
Note that for each $h\in N$, $\gamma\in\Gamma$ and $x\in B\left(h\gamma\left(0\right),R_{G}\right)$
we have
\[
d\left(h\left(0\right),x\right)\le d\left(h\left(0\right),h\gamma\left(0\right)\right)+d\left(h\gamma\left(0\right),x\right)<\max\left\{ d\left(0,\gamma\left(0\right)\right):\gamma\in\Gamma\right\} +R_{G},
\]
which implies that
\[
\bigcup_{h\in N}\bigcup_{\gamma\in\Gamma}B\left(h\gamma\left(0\right),R_{G}\right)\subset\bigcup_{h\in N}B\left(h\left(0\right),R_{G}+\max\left\{ d\left(0,\gamma\left(0\right)\right):\gamma\in\Gamma\right\} \right).
\]
Thus, $\xi\in L_{\mathrm{ur}}\left(N\right)$.

For the converse inclusion, let $\xi\in L_{\mathrm{ur}}\left(N\right)$ be given. Then, by the definition of $L_{\mathrm{ur}}\left(N\right)$,
there exists a constant $c:=c\left(\xi\right)>0$ such that
\[
s_{\xi}\subset\bigcup_{h\in N}B\left(h\left(0\right),c\right).
\]
Hence, by (\ref{eq:Lur-coding-2}), there exists a finite set $\Gamma\subset G$
such that $s_{\xi}\subset\bigcup_{h\in N}\bigcup_{\gamma\in\Gamma}h\gamma\overline{F}$. We conclude that for $\omega=\left(v_{k},w_{k}\right)_{k\in\N}\in\Sigma_{\Phi_G}$
with $\pi_{\Phi_G}\left(\omega\right)=\xi$ we have that $\left\{ v_{1}v_{2}\cdots v_{k}:\,\, k\in\N\right\} \subset N\Gamma$
and hence, $\xi\in \Lur\left(N,\Phi_G\right)$.

Let us now address the inclusion $\Lr\left(N,\Phi_G\right)\subset L_{\mathrm{r}}\left(N\right)$. For this, let $\xi\in \Lr\left(N,\Phi_G\right)$ be given. By the definition of $\Lr\left(N,\Phi_G\right)$,
there exists $\omega=\left(v_{k},w_{k}\right)_{k\in\N}\in\Sigma_{\Phi_G}$, an element
$\gamma\in G$, a sequence $\left(h_{k}\right)_{k\in\N}$ of pairwise distinct  elements
in $N$ and a sequence $\left(n_{k}\right)_{k\in\N}$ tending to infinity
such that $\pi_{\Phi_G}\left(\omega\right)=\xi$ and $v_{1}v_{2}\cdots v_{n_{k}}=h_{k}\gamma$,
for all $k\in\N$. Using  (\ref{eq:schottky-ex-2}) it follows that $s_{\xi}\cap B\left(h_{k}\gamma\left(0\right),R_{G}\right)\neq\emptyset$, for all $k\in \N$.
Since  $B\left(h_{k}\gamma\left(0\right),R_{G}\right)\subset B\left(h_{k}\left(0\right),R_{G}+d\left(0,\gamma\left(0\right)\right)\right)$
for all $k\in\N$, we obtain that also $s_{\xi}\cap B\left(h_{k}\left(0\right),R_{G}+d\left(0,\gamma\left(0\right)\right)\right)\neq\emptyset$. We have thus shown that $\xi\in L_{\mathrm{r}}\left(N\right)$.

Finally, let us demonstrate that $L_{\mathrm{r}}\left(N\right)\subset \Lr\left(N,\Phi_G\right)$. To that end, pick an arbitrary $\xi\in L_{\mathrm{r}}\left(N\right)$ and let $\omega=\left(v_{k},w_{k}\right)_{k\in\N}\in\Sigma_{\Phi_G}$
with  $\pi_{\Phi_G}\left(\omega\right)=\xi$ be given. Then,  by definition of $L_{\mathrm{r}}\left(N\right)$,
there exists $c>0$ and a sequence $\left(h_{k}\right)_{k\in\N}$
of pairwise distinct elements in $N$ such that $s_{\xi}\cap B\left(h_{k}\left(0\right),c\right)\neq\emptyset$,
for all $k\in\N$. Using (\ref{eq:Lur-coding-2}) we deduce that  there exists a finite
set $\Gamma\subset G$ such that for all $k\in\N$ we have
\[
s_{\xi}\cap B\left(h_{k}\left(0\right),c\right)\cap\bigcup_{\gamma\in\Gamma}h_{k}\gamma \overline{F}\neq\emptyset.
\]
Since $\Gamma$ is finite, there exist  $\gamma_0\in\Gamma$ and
 sequences $\left(n_{k}\right)_{k\in \N}$  and $\left(l_{k}\right)_{k\in \N}$ tending to infinity such
that $s_{\xi}\cap B\left(h_{n_{k}}\left(0\right),c\right)\cap h_{n_{k}}\gamma_0 \overline{F}\neq\emptyset$ and $v_{1}v_{2}\cdots v_{l_{k}}=h_{n_{k}}\gamma_0$, for all $k\in \N$. Hence,  $\xi\in \Lr\left(N,\Phi_G\right)$.
\end{proof}

\providecommand{\bysame}{\leavevmode\hbox to3em{\hrulefill}\thinspace}
\providecommand{\MR}{\relax\ifhmode\unskip\space\fi MR }
\providecommand{\MRhref}[2]{%
  \href{http://www.ams.org/mathscinet-getitem?mr=#1}{#2}
}
\providecommand{\href}[2]{#2}


\begin{thebibliography}{BTMT12}

\bibitem[AD00]{MR1803461}
J.~Aaronson and M.~Denker, \emph{On exact group extensions}, Sankhy\=a Ser. A
  \textbf{62} (2000), no.~3, 339--349, Ergodic theory and harmonic analysis
  (Mumbai, 1999). \MR{1803461 (2001m:37011)}

\bibitem[AD02]{MR1906436}
\bysame, \emph{Group extensions of {G}ibbs-{M}arkov maps}, Probab. Theory
  Related Fields \textbf{123} (2002), no.~1, 38--40. \MR{1906436 (2003f:37010)}

\bibitem[Bea95]{Beardon}
A.~F. Beardon, \emph{The geometry of discrete groups}, Graduate Texts in
  Mathematics, vol.~91, Springer-Verlag, New York, 1995, Corrected reprint of
  the 1983 original. \MR{1393195 (97d:22011)}

\bibitem[Ber98]{MR1650275}
G.~M. Bergman, \emph{An invitation to general algebra and universal
  constructions}, Henry Helson, Berkeley, CA, 1998. \MR{1650275 (99h:18001)}

\bibitem[BJ97]{MR1484767}
C.~J. Bishop and P.~W. Jones, \emph{Hausdorff dimension and {K}leinian groups},
  Acta Math. \textbf{179} (1997), no.~1, 1--39. \MR{1484767 (98k:22043)}

\bibitem[BM74]{MR0333164}
A.~F. Beardon and B.~Maskit, \emph{Limit points of {K}leinian groups and finite
  sided fundamental polyhedra}, Acta Math. \textbf{132} (1974), 1--12.
  \MR{0333164 (48 \#11489)}

\bibitem[Bow75]{bowenequilibriumMR0442989}
R.~Bowen, \emph{Equilibrium states and the ergodic theory of {A}nosov
  diffeomorphisms}, Springer-Verlag, Berlin, 1975, Lecture Notes in
  Mathematics, Vol. 470. \MR{MR0442989 (56 \#1364)}

\bibitem[Bro85]{MR783536}
R.~Brooks, \emph{The bottom of the spectrum of a {R}iemannian covering}, J.
  Reine Angew. Math. \textbf{357} (1985), 101--114. \MR{783536 (86h:58138)}

\bibitem[BTMT12]{Bonfert-Taylor2012}
P.~Bonfert-Taylor, K.~Matsuzaki, and E.~C. Taylor, \emph{Large and small covers
  of a hyperbolic manifold}, J. Geom. Anal. \textbf{22} (2012), no.~2,
  455--470.

\bibitem[Coh82]{MR678175}
J.~M. Cohen, \emph{Cogrowth and amenability of discrete groups}, J. Funct.
  Anal. \textbf{48} (1982), no.~3, 301--309. \MR{MR678175 (85e:43004)}

\bibitem[Day49]{day1949amenabledef}
M.~M. Day, \emph{Means on semigroups and groups}, Bull. Amer. Math. Soc.
  \textbf{55} (1949), 1054--1055.

\bibitem[Day64]{MR0159230}
M.~M. Day, \emph{Convolutions, means, and spectra}, Illinois J. Math.
  \textbf{8} (1964), 100--111. \MR{0159230 (28 \#2447)}

\bibitem[DK86]{MR894523}
J.~Dodziuk and W.~S. Kendall, \emph{Combinatorial {L}aplacians and
  isoperimetric inequality}, From local times to global geometry, control and
  physics ({C}oventry, 1984/85), Pitman Res. Notes Math. Ser., vol. 150,
  Longman Sci. Tech., Harlow, 1986, pp.~68--74. \MR{894523 (88h:58118)}

\bibitem[Dod84]{MR743744}
J.~Dodziuk, \emph{Difference equations, isoperimetric inequality and transience
  of certain random walks}, Trans. Amer. Math. Soc. \textbf{284} (1984), no.~2,
  787--794. \MR{743744 (85m:58185)}

\bibitem[F{\o}l55]{MR0079220}
E.~F{\o}lner, \emph{On groups with full {B}anach mean value}, Math. Scand.
  \textbf{3} (1955), 243--254. \MR{0079220 (18,51f)}

\bibitem[FS04]{MR2097162}
K.~Falk and B.~O. Stratmann, \emph{Remarks on {H}ausdorff dimensions for
  transient limit sets of {K}leinian groups}, Tohoku Math. J. (2) \textbf{56}
  (2004), no.~4, 571--582. \MR{2097162 (2005g:30053)}

\bibitem[Ger88]{MR938257}
P.~Gerl, \emph{Random walks on graphs with a strong isoperimetric property}, J.
  Theoret. Probab. \textbf{1} (1988), no.~2, 171--187. \MR{938257 (89g:60216)}

\bibitem[Gui80]{MR588157}
Y.~Guivarc'h, \emph{Sur la loi des grands nombres et le rayon spectral d'une
  marche al\'eatoire}, Conference on {R}andom {W}alks ({K}leebach, 1979)
  ({F}rench), Ast\'erisque, vol.~74, Soc. Math. France, Paris, 1980,
  pp.~47--98, 3. \MR{588157 (82g:60016)}

\bibitem[Jae11]{JaerischDissertation11}
J.~Jaerisch, \emph{Thermodynamic formalism for group-extended {M}arkov
  systems with applications to {F}uchsian groups}, Doctoral dissertation
  at the University Bremen (2011).

\bibitem[Jae12]{Jaerisch12a}
\bysame, \emph{A lower bound for the exponent of convergence of normal
  subgroups of {K}leinian groups}, arXiv:1203.3022v1 (2012).

\bibitem[JKL10]{JaerischKessebohmer10}
J.~Jaerisch, M.~Kesseb{\"o}hmer, and S.~Lamei, \emph{Induced topological
  pressure for countable state {M}arkov shifts}, arXiv:1010.2162v1 (2010).

\bibitem[Kai92]{MR1245225}
V.~A. Kaimanovich, \emph{Dirichlet norms, capacities and generalized
  isoperimetric inequalities for {M}arkov operators}, Potential Anal.
  \textbf{1} (1992), no.~1, 61--82. \MR{1245225 (94i:31012)}

\bibitem[Kes59a]{MR0112053}
H.~Kesten, \emph{Full {B}anach mean values on countable groups}, Math. Scand.
  \textbf{7} (1959), 146--156. \MR{MR0112053 (22 \#2911)}

\bibitem[Kes59b]{MR0109367}
\bysame, \emph{Symmetric random walks on groups}, Trans. Amer. Math. Soc.
  \textbf{92} (1959), 336--354. \MR{MR0109367 (22 \#253)}

\bibitem[Koo31]{koopman31}
B.~O. Koopman, \emph{Hamiltonian systems and transformation in {H}ilbert
  space}, Proceedings of the National Academy of Sciences of the United States
  of America, vol.~17, 1931.

\bibitem[LM94]{MR1244104}
A.~Lasota and M.~C. Mackey, \emph{Chaos, fractals, and noise}, second ed.,
  Applied Mathematical Sciences, vol.~97, Springer-Verlag, New York, 1994,
  Stochastic aspects of dynamics. \MR{1244104 (94j:58102)}

\bibitem[Mas88]{MR959135}
B.~Maskit, \emph{Kleinian groups}, Grundlehren der Mathematischen
  Wissenschaften [Fundamental Principles of Mathematical Sciences], vol. 287,
  Springer-Verlag, Berlin, 1988. \MR{959135 (90a:30132)}

\bibitem[Moh88]{MR943998}
B.~Mohar, \emph{Isoperimetric inequalities, growth, and the spectrum of
  graphs}, Linear Algebra Appl. \textbf{103} (1988), 119--131. \MR{943998
  (89k:05071)}

\bibitem[MT98]{MR1638795}
K.~Matsuzaki and M.~Taniguchi, \emph{Hyperbolic manifolds and {K}leinian
  groups}, Oxford Mathematical Monographs, The Clarendon Press Oxford
  University Press, New York, 1998, Oxford Science Publications. \MR{1638795
  (99g:30055)}

\bibitem[MU03]{MR2003772}
R.~D. Mauldin and M.~Urba{\'n}ski, \emph{Graph directed {M}arkov systems},
  Cambridge Tracts in Mathematics, vol. 148, Cambridge, 2003.

\bibitem[Neu29]{vonNeumann1929amenabledef}
J.~v. Neumann, \emph{{Zur allgemeinen Theorie des Masses}}, Fund. Math.
  \textbf{13} (1929), 73--116.

\bibitem[Nic89]{MR1041575}
P.~J. Nicholls, \emph{The ergodic theory of discrete groups}, London
  Mathematical Society Lecture Note Series, vol. 143, Cambridge University
  Press, Cambridge, 1989. \MR{1041575 (91i:58104)}

\bibitem[OW07]{MR2338235}
R.~Ortner and W.~Woess, \emph{Non-backtracking random walks and cogrowth of
  graphs}, Canad. J. Math. \textbf{59} (2007), no.~4, 828--844. \MR{MR2338235
  (2008h:05057)}

\bibitem[P{\'o}l21]{MR1512028}
G.~P{\'o}lya, \emph{\"{U}ber eine {A}ufgabe der {W}ahrscheinlichkeitsrechnung
  betreffend die {I}rrfahrt im {S}tra\ss ennetz}, Math. Ann. \textbf{84}
  (1921), no.~1-2, 149--160. \MR{1512028}

\bibitem[Rat06]{MR2249478}
J.~G. Ratcliffe, \emph{Foundations of hyperbolic manifolds}, second ed.,
  Graduate Texts in Mathematics, vol. 149, Springer, New York, 2006.
  \MR{2249478 (2007d:57029)}

\bibitem[Rob05]{MR2166367}
T.~Roblin, \emph{Un th\'eor\`eme de {F}atou pour les densit\'es conformes avec
  applications aux rev\^etements galoisiens en courbure n\'egative}, Israel J.
  Math. \textbf{147} (2005), 333--357. \MR{2166367 (2006i:37065)}

\bibitem[RU08]{MR2413348}
M.~Roy and M.~Urba{\'n}ski, \emph{Real analyticity of {H}ausdorff dimension for
  higher dimensional hyperbolic graph directed {M}arkov systems}, Math. Z.
  \textbf{260} (2008), no.~1, 153--175. \MR{2413348 (2009m:37064)}

\bibitem[Rud73]{MR0365062}
W. Rudin, \emph{Functional analysis}, McGraw-Hill Series in Higher Mathematics, New York, 1973.  \MR{365062}


\bibitem[Rue69]{MR0289084}
D.~Ruelle, \emph{Statistical mechanics: {R}igorous results}, W. A. Benjamin,
  Inc., New York-Amsterdam, 1969. \MR{0289084 (44 \#6279)}

\bibitem[Sen06]{MR2209438}
E.~Seneta, \emph{Non-negative matrices and {M}arkov chains}, Springer Series in
  Statistics, Springer, New York, 2006, Revised reprint of the second (1981)
  edition [Springer-Verlag, New York; MR0719544]. \MR{2209438}

\bibitem[Sha07]{MR2322540}
R.~Sharp, \emph{Critical exponents for groups of isometries}, Geom. Dedicata
  \textbf{125} (2007), 63--74. \MR{2322540 (2008f:20107)}

\bibitem[Sta13]{Stadlbauer11}
M.~Stadlbauer, \emph{An extension of {K}esten's criterion for amenability to topological {M}arkov chains}, 
 Adv. Math. \textbf{235} (2013), 450--468. 


\bibitem[Str04]{MR2087134}
B.~O. Stratmann, \emph{The exponent of convergence of {K}leinian groups; on a
  theorem of {B}ishop and {J}ones}, Fractal geometry and stochastics {III},
  Progr. Probab., vol.~57, Birkh\"auser, Basel, 2004, pp.~93--107. \MR{2087134
  (2005h:20114)}

\bibitem[Str06]{MR2191250}
\bysame, \emph{Fractal geometry on hyperbolic manifolds}, Non-{E}uclidean
  geometries, Math. Appl. (N. Y.), vol. 581, Springer, New York, 2006,
  pp.~227--247. \MR{2191250 (2006g:37038)}

\bibitem[VJ62]{MR0141160}
D.~Vere-Jones, \emph{Geometric ergodicity in denumerable {M}arkov chains},
  Quart. J. Math. Oxford Ser. (2) \textbf{13} (1962), 7--28. \MR{0141160 (25
  \#4571)}

\bibitem[Wal82]{MR648108}
P.~Walters, \emph{An introduction to ergodic theory}, Graduate Texts in
  Mathematics, vol.~79, Springer-Verlag, New York, 1982. \MR{MR648108
  (84e:28017)}

\bibitem[Woe00]{MR1743100}
W.~Woess, \emph{Random walks on infinite graphs and groups}, Cambridge Tracts
  in Mathematics, vol. 138, Cambridge University Press, Cambridge, 2000.
  \MR{1743100 (2001k:60006)}

\end{thebibliography}
\end{document}